\documentclass[a4paper,11pt,reqno,noindent]{amsart}
\usepackage[centertags]{amsmath}
\usepackage{amsfonts,amssymb,amsthm,amscd,dsfont,cases,esint,enumerate,color}
\usepackage[T1]{fontenc}
\usepackage[english]{babel}
\usepackage[applemac]{inputenc}
\usepackage[body={15cm,21.5cm},centering]{geometry} 
\usepackage{fancyhdr}
\numberwithin{equation}{section}

\theoremstyle{plain}
\newtheorem{theor10}{Theorem}

\newtheorem{prop10}[theor10]{Proposition}

\newtheorem{cor10}[theor10]{Corollary}

\newtheorem{lem10}[theor10]{Lemma}

\newtheorem{theor0}{Theorem}[section]
\newenvironment{theor}
  {\pushQED{\qed}\begin{theor0}}
  {\popQED\end{theor0}}
\newtheorem{lem0}[theor0]{Lemma}
\newenvironment{lem}
  {\pushQED{\qed}\begin{lem0}}
  {\popQED\end{lem0}}
\newtheorem{prop0}[theor0]{Proposition}
\newenvironment{prop}
  {\pushQED{\qed}\begin{prop0}}
  {\popQED\end{prop0}}
\newtheorem{cor0}[theor0]{Corollary}

\theoremstyle{definition}
\newtheorem{rems0}[theor0]{Remarks}

\newtheorem{rem0}[theor0]{Remark}
\newenvironment{rem}
  {\pushQED{\qed}\begin{rem0}}
  {\popQED\end{rem0}}
 \newtheorem{hypo0}[theor0]{Hypothesis}
\newenvironment{hypo}
  {\pushQED{\qed}\begin{hypo0}}
  {\popQED\end{hypo0}}

\theoremstyle{plain}
\newtheorem{as0}[theor0]{Assumption}

\newtheorem*{asn0*}{\assumptionnumber}
  \providecommand{\assumptionnumber}{}
\makeatletter
\newenvironment{asn0}[2]
   {\renewcommand{\assumptionnumber}{Assumption \!#1 {\normalfont--- #2}}
    \begin{asn0*}
    \protected@edef\@currentlabel{{\normalfont#1}}}
   {\end{asn0*}}
\makeatother

\makeatletter
\newenvironment{asn01}[1]
   {\renewcommand{\assumptionnumber}{Assumption \!#1}
    \begin{asn0*}
    \protected@edef\@currentlabel{{\normalfont#1}}}
   {\end{asn0*}}
\makeatother

\newcommand{\N}{\mathbb N}

\newcommand{\e}{\varepsilon}

\newcommand{\calQ}{{\mathcal Q}}
\newcommand{\calC}{\mathcal C}
\newcommand{\Nc}{\mathcal N}

\newcommand{\mone}{{\mathds{1}}}

\newcommand{\R}{\mathbb R}
\newcommand{\Z}{\mathbb Z}

\newcommand{\loc}{{\operatorname{loc}}}
\newcommand{\Id}{\operatorname{Id}}

\newcommand{\step}[1]{\noindent \textit{Step} #1.}

\newcommand{\expec}[1]{\mathbb{E}\left[ #1 \right]}

\newcommand{\var}[1]{\mathrm{Var}\left[#1\right]}
\newcommand{\cov}[2]{\mathrm{Cov}\left[#1;#2\right]}
\newcommand{\dTV}[2]{d_{TV}\left(#1,#2\right)}

\newcommand{\ent}[1]{\mathrm{Ent}\!\left[#1\right]}

\newcommand{\3}{\operatorname{|\!|\!|}}

\newcommand{\sign}[1]{\operatorname{sign} \,#1}

\usepackage[colorlinks,citecolor=black,urlcolor=black]{hyperref}

\title[Quantitative homogenization for log-normal coefficients in dimension one]{Quantitative homogenization for log-normal coefficients via Malliavin calculus:
the one-dimensional case}

\author[A. Gloria]{Antoine Gloria}
\address[Antoine Gloria]{Sorbonne Universit\'e, CNRS, Universit\'e de Paris, Laboratoire Jacques-Louis Lions, 75005~Paris, France \& Institut Universitaire de France \& Universit\'e Libre de Bruxelles, D\'epartement de Math\'ematique, 1050~Brussels, Belgium}
\email{antoine.gloria@sorbonne-universite.fr}

\author[S. Qi]{Siguang Qi}
\address[Siguang Qi]{Sorbonne Universit\'e, CNRS, Universit\'e de Paris, Laboratoire Jacques-Louis Lions, 75005~Paris, France }
\email{siguang.qi@sorbonne-universite.fr}

\begin{document}
\selectlanguage{english}

\maketitle

\begin{abstract}
The quantitative analysis of stochastic homogenization problems has been a very active field
in the last fifteen years.  Whereas the first results were motivated by applied questions (namely, the numerical approximation of homogenized coefficients), the more recent  achievements in the field are much more analytically-driven and focus on the subtle interplay between PDE analysis (and in particular elliptic regularity theory) and probability (concentration, stochastic cancellations, scaling limits). 
The aim of this article is threefold. 
First we provide a complete and self-contained analysis for the popular example of log-normal coefficients with possibly fat tails in dimension $d=1$, establishing new results on the accuracy of the two-scale expansion and characterizing fluctuations (in the perspective of uncertainty quantification). Second, we work in a context where explicit formulas allow us to by-pass analytical difficulties and therefore mostly focus on the probabilistic side of the theory.
Last, the one-dimensional setting gives intuition on the available results in higher dimension (provided results are correctly reformulated) to which we give precise entries to the recent literature. 

\bigskip\noindent
{\sc MSC-class:} 35R60, 35B27, 35B65, 60H07

\end{abstract}

\setcounter{tocdepth}{1}
\tableofcontents

\section{Introduction}

\subsection{Position of the problem}

Consider for a moment a random checkerboard $a:\R\to \{1,2\}$, namely, a random field on $\R$ that is piecewise constant on intervals of the form $[z,z+1)$, $z\in \Z$,
and the values of which are picked independently on each interval by tossing identical unbiased coins (head is $1$, tail is $2$).
Let $f \in C^1([0,1])$, and for all $\e >0$, consider the equation posed on $(0,1)$
\begin{equation}\label{0:1D-1}
(a(\frac x  \e)u_\e'(x))'=f'(x), \quad u_\e(0)=0, \quad u_\e(1)=0.
\end{equation}
The aim of stochastic homogenization is to describe the asymptotic behavior of the unique solution $u_\e \in C^1([0,1])$ of \eqref{0:1D-1} when $\e$ gets small.

\medskip

A direct integration of \eqref{0:1D-1} (see \eqref{0:1D-2}  below) allows to express $u_\e$ in terms of
integrals of $\frac 1a$. In order to motivate the questions we shall address, let us consider 
such integrals of random fields $a$ itself. The law of large numbers ensures
that almost surely
\begin{equation}\label{0:baby1}
\lim_{N\uparrow \infty} \frac1N \int_0^N a(x)dx \,=\, \expec{\int_0^1a}=\frac12(1+2)=\frac32.
\end{equation}
This proves that the random field $\frac1N \int_0^N a(x)dx$ has a deterministic limit. 
At which speed does this convergence take place? 
This question turns out to be twofold. Since $\frac1N \int_0^N a(x)dx$ is random,  $\frac1N \int_0^N a(x)dx$ fluctuates
around its expectation $\expec{\frac1N \int_0^N a(x)dx}$. What \eqref{0:baby1} shows is that $\expec{\frac1N \int_0^N a(x)dx}$ converges
to $\frac32$ as $N\uparrow \infty$ and that the fluctuations of $\frac1N \int_0^N a(x)dx-\expec{\frac1N \int_0^N a(x)dx}$ vanish as $N\uparrow \infty$. To understand the speed of 
convergence we thus need to establish the convergence of the expectation (this amounts to establishing an error estimate
for the convergence of a deterministic sequence: the sequence of expectations), and to control the fluctuations.
In this simple example, we have $\expec{\frac1N \int_0^N a(x)dx}=\expec{\int_0^1a}=\frac32$, so that there is 
no ``systematic error'' (that is, the sequence of expectations is constant).
However, we do have fluctuations, and therefore a ``random error''.
A direct calculation using the independence of the values of $a$ on unit intervals shows that
\begin{eqnarray*}
\lefteqn{\var{ \frac1N \int_0^N a(x)dx}}
\\
&=& \expec{\Big(\frac1N \int_0^N a(x)dx-\frac32\Big)^2}= \frac1{N^2}\int_0^N\int_0^N \expec{(a(x)-\frac32)(a(y)-\frac32)}dxdy
\\ &=& \frac1N \expec{\frac1N \int_0^N {(a(x)-\frac32)^2dx}} 
\,=\, \frac1N \frac14.
\end{eqnarray*}
Hence one expects $ \frac1{\sqrt{N}} \int_0^N (a(x)-\frac32)dx$ (note the scaling factor $\frac1{\sqrt N}$ instead of $\frac1N$) to be of order $1$ with high probability. 
How can we make this statement precise?
A possible way is to investigate whether the sequence of random variables $\frac1{\sqrt{N}} \int_0^N (a(x)-\frac32)dx$ admits a limiting law as $N\uparrow \infty$ (or in other words whether the distribution of the random variable $\frac1{\sqrt{N}} \int_0^N (a(x)-\frac32)dx$ converges towards a limiting distribution).
In our setting of a simple random checkerboard (that reduces to tossing coins), the answer is given by the central limit theorem
which ensures that for large $N$, the law of $\frac1{\sqrt{N}} \int_0^N (a(x)-\frac32)dx$ is close to the law of a Gaussian of mean $0$ and  variance $\frac14$ --- which
we call a centered normal random variable of variance $\frac14$.

\medskip

In stochastic homogenization, $x\mapsto u'_\e( \e x)$ plays a similar role as $a$, and $\frac1\e$ the role of  $N$.
The first question is whether spatial averages of $u'_\e$ become deterministic as $\e \downarrow 0$, which constitutes a qualitative ergodic theorem.
This requires to identify the limit, which we shall call $\bar u'$.
The second question is whether this convergence can be quantified, both in terms of systematic and random errors. And finally, 
the third question concerns the validity of a central limit theorem for averages of $u'_\e$.
Answering these questions is the aim of the quantitative theory of stochastic homogenization.

\subsection{Aim and outline of the article}

The quantitative analysis of stochastic homogenization problems has been a very active field
in the last fifteen years.  Whereas the first results on weakly correlated fields (such as the random checkerboard or Poisson random inclusions) \cite{GO1,GO2,Gloria-Otto-10b,GNO1,GNO2} were motivated by applied questions (namely, the numerical approximation of homogenized coefficients), the more recent and powerful achievements in the field are much more analytically-driven and focus on the subtle interplay between PDE analysis (and in particular elliptic regularity theory) and probability (concentration, stochastic cancellations, scaling limits) --- see \cite{DG1,DG2,GNO-reg,GNO-quant,DGO1,DGO2,DO-20,duerinckx2019scaling} for approaches based on functional calculus and nonlinear concentration of measure, and \cite{AS,Armstrong-Mourrat-16,AKM1,GO4,AKM2,AKM-book} for approaches based on renormalization and linear concentration of measure.
The quantitative stochastic homogenization theory --- with its many probabilistic refinements --- has drifted away from applications, and somehow eludes the PDE community of homogenization.

\medskip

This article is an attempt to treat a problem of practical interest (namely, the popular log-normal coefficients used in geophysics -- e.g.~\cite{MR1455261} and the references therein) in a way that allows one to   focus on probabilistic aspects of the theory in the simplest context possible.
To this aim we propose to follow an unusual path, and split this contribution into three parts.
In the first part (Section~\ref{sec:Malliavin}), we introduce the most versatile probabilistic tool of quantitative stochastic homogenization from a practical viewpoint: Malliavin and Stein-Malliavin calculus. In the second part, we focus on the one-dimensional setting, for which we have explicit formulas (Section~\ref{sec:qual}). This allows us to bypass PDE analysis, and use Malliavin calculus on concrete objects. We do this for coefficients that follow  a log-normal law (with possibly fat tails), and display a complete analysis, characterizing both the oscillations (Section~\ref{sec:2s}) and the fluctuations (Section~\ref{sec:law}) of the solutions.  
The use of explicit formulas is obviously limited to dimension $d=1$ and the proofs we present cannot extend to higher dimension.  The results we prove, however, are representative of what happens in higher dimension when carefully reformulated. 
In the third part of this article, we therefore revisit these one-dimensional results with the crucial concepts of homogenization in mind (such as corrector, two-scale expansion, homogenization commutator, homogenized coefficients, and fluctuation tensor, that we precisely define),  in a way that can be extended to higher dimension (for which we give precise entries to the literature).

\section{Malliavin calculus in practice}\label{sec:Malliavin}

In this short section, we provide the tools we shall borrow from Malliavin calculus.
We focus the presentation on the practice: Malliavin calculus is first of all a calculus, which allows to prove actual estimates. In particular, 
we deliberately ignore the functional-analytic setting. We refer the reader to \cite{duerinckx2019scaling} 
for a gentle introduction of this functional-analytic setting and to precise references to 
standard monographs on Malliavin calculus. In the following, we favor simplicity to generality.
Let $G:\R \to \R$ be a stationary Gaussian random field of mean zero with covariance function $\calC:\R\to \R$ (which we assume to be H\"older continuous at zero) that satisfies for some $\beta>0$ and all $x\in \R$
$$
|\calC(x)| \lesssim (1+|x|)^{-\beta}.
$$
Recall that a Gaussian random variable is characterized by its mean and its variance.
In particular, the above means that for all $x \in \R$, $G(x)$ is a centered Gaussian random variable of variance $\calC(0)$, and that for all $x,y  \in \R$, $G(x)$ and $G(y)$ are two Gaussian random variables whose covariance $\expec{G(x)G(y)}$ is given by $\calC(x-y)$ (if $\calC(x-y)\ne 0$, $G(x)$ and $G(y)$ are not independent).
We further assume that $\calC = \calC^\circ \star \calC^\circ$ for some $\calC^\circ$ satisfying for all $x\in \R$
\begin{equation} \label{e.decay-Co}
\sup_{B(x)} |\calC^\circ | \,\lesssim \, (1+|x|)^{-1-\frac \beta 2}\times
\left\{
\begin{array}{lll}
\log (2+|x|)^{-\frac12}&:&\beta=1,\\
1&:&\beta\ne 1.
\end{array}
\right.
\end{equation}
(While these assumptions could be weakened, they are rather general and yield substantial simplifications for our purposes.)

The Gaussian field $G$ will play the role of what one usually abstractly writes ``$\omega$'' in stochastic homogenization. In particular, a random variable will be nothing but a measurable map $X$ of $G$ (let us be loose on measurability and functional analysis). 
Since $G$ is a Gaussian field, for all compactly supported continuous functions $h$, the random variable $X_h(G):=\int_{\R} h(x)G(x)dx$
is itself Gaussian, and thus characterized by its mean $\expec{X_h}=0$ and its variance 
$\var{X_h}=\expec{X_h^2}:=\iint_{\R\times \R}h(x)\calC(x-y)h(y)dxdy$. 
In what follows, we thus abusively use the notation $\Omega$ for the set of realizations
of the Gaussian field $G$, and shall use the suggestive notation $G$ instead of the classical notation $\omega$.

\subsection{Stationarity and the ergodic theorem}\label{sec:erg-th}

We start with the definition of stationarity.
We say that a random field $F:\R \times \Omega \to \R$ (that is, a function  both of space and of the Gaussian field) is stationary if 
for all $x,z\in \R$ and $G\in \Omega$ we have $F(x+z,G)=F(x,G(\cdot+z))$ (that is, translating in space amounts to translating the Gaussian field -- the Gaussian field is defined on $\R$).
Since 
\[
\expec{F(x)}:=\int_\R F(x,t) g(t)dt,
\] 
where $g(t):=\frac1{\sqrt{2\pi \calC(0)}} e^{-\frac{t^2}{2\calC(0)}}$, 
and since $G(\cdot+z)$ and $G$ have the same density $g$, 
the expectation $\expec{F(x)}$ of $F(x)$ does not depend on the point $x$ -- it is denoted by $\expec{F}$.
If $F \in L^1_\loc(\R,L^1(\Omega))$, then for all $h\in L^\infty(\R)$ with compact support,
we have for almost every $G$
\begin{equation*}
\lim_{\e \downarrow 0} \int_{\R} F(\tfrac x\e,G) h \,=\,  \expec{F} \int_{\R} h,
\end{equation*}
which we refer to as the ergodic theorem.
Of course this convergence (which actually only requires the ergodicity of the probability measure) will be very much improved in our Gaussian setting. Note that the ergodic theorem is more often presented in the form
\[
\lim_{R \uparrow +\infty} \fint_{-R}^R F(x,G) dx \,=\,  \expec{F},
\]
which is completely equivalent (by a change of variable, and an approximation argument).

\subsection{The Malliavin derivative}

Let $X$ be a random variable. 
For all $x\in \R^d$, we define 
\begin{equation}\label{e.Mallder}
D_x X(G)\,:=\, \frac{\partial X(G)}{\partial G(x)},
\end{equation}
where $\frac{\partial }{\partial G(x)}$ is the standard derivative with respect to the value of the Gaussian field at $x$.
Let us give an example, and consider the random variable $X_h(G):=\int_{\R} h(x)G(x)dx$ for some smooth compactly supported function $h$. Then 
\[
D_z X_h(G)=\int_{\R} h(x)\frac{\partial G(x)}{\partial G(z)}dx=\int_{\R} h(x)\delta(x-z)dx=h(z).
\]
The Malliavin derivative $DX(G)$ is the collection  $\{D_x X(G)\}_{x\in \R}$ (it is well-defined as an element of a suitable Malliavin-Sobolev space constructed using the covariance function $\calC$).
Likewise, for all $x,x'\in \R^d$, we define 
$$
D^2_{x,x'} X(G)\,: = \, \frac{\partial^2 X(G)}{\partial G(x)\partial G(x')},
$$
where $\frac{\partial^2}{\partial G(x)\partial G(x')}$ is the standard second derivative with respect to the value of the Gaussian field at $x$ and $x'$. In our linear example of $X_h(G):=\int_{\R} h(x)G(x)dx$, we obviously have 
\[
D^2_{z,z'} X_h(G)=D_{z} h(z')= 0
\]
(the object $h(z')$ is deterministic and does not depend on $G$).
Likewise, the second Malliavin derivative $D^2X(G)$  is the collection  $\{D^2_{x,x'} X(G)\}_{x,x'\in \R}$ (it is well-defined as an element of a suitable Malliavin-Sobolev space).

\subsection{Control of the distance to a constant: Moment bounds}
The ergodic theorem shows that the fluctuations of a sequence of random variables (averages on larger and larger sets) vanish at the limit. To quantify fluctuations, we can consider variances (or more generally, $L^p$ norms of the random variables minus their expectation). This is what Malliavin calculus can do for us.
To this aim, we first introduce a suitable norm on maps $\R \to \R$ which only depends on the exponent $\beta$ of the decay of $\calC^\circ$: For all functions $f:\R\to \R$, we set
\begin{equation}\label{e.Mallnorm}
\|f\|_\beta  := \begin{cases}
\|f\|_{ L^\frac{2}{2-\beta}(\R)},&\text{if $0<\beta<1$};\\
\|\log(2+|\cdot|)^\frac12f\|_{ L^{2}(\R)},&\text{if $\beta=1$};\\
\|f\|_{ L^2(\R)},&\text{if $\beta>1$}.
\end{cases}
\end{equation}
For $\beta >1$, we denote $\|\cdot\|_\beta$ by $\|\cdot\|_{1+}$.
(The definitions of these norms are specific to dimension $d=1$, see \cite{duerinckx2019scaling} in  higher dimension.)
The first-order Malliavin derivative can be used to control how far a random variable $X$ is from its expectation $\expec{X}$ (basically, a Poincar\'e inequality in the ``infinite-dimensional'' probability space). 
The strongest such control is given by the following logarithmic-Sobolev inequality (LSI) for the Malliavin derivative: For all random variables $X$, we have
\begin{equation}\label{e.Mal:LSI}
\ent{X^2}:=\expec{X^2\log\frac{X^2}{\expec{X^2}}}\lesssim \,\expec{\|DX\|_\beta^2}.
\end{equation}
We shall not use (LSI) in the form of \eqref{e.Mal:LSI}, but rather in form of the moment bounds it implies: For all $p \ge 1$, 
\begin{equation}\label{e.Mal:moment}
\expec{(X-\expec{X})^{2p}}^\frac1{2p} \lesssim \sqrt p \, \expec{\|DX\|_\beta^{2p}}^\frac{1}{2p}.
\end{equation}
This is what is called a Poincar\'e inequality in $W^{1,p}(\Omega)$.

\subsection{Formula for covariances}

Whereas \eqref{e.Mal:moment} provides a control of the variance from above, it does not allow to control it from below (which is necessary to prove that the limiting variance is non zero in scaling limits -- we will define these words later). To this aim, we recall an identity for the covariance of two random variables $X$ and $Y$, known as the Helffer-Sj\"ostrand representation formula. Of course, since this is an identity, it makes use of the whole covariance function $\mathcal C$, and not only of the value of $\beta$.
We have
\begin{equation}
\cov{X}{Y}\,=\,  \expec{\iint_{\R \times \R} D_xX \, \mathcal C(x-y) (1+\mathcal L)^{-1} D_yYdxdy},
\end{equation}
where $(1+\mathcal L)^{-1} $ is a self-adjoint linear operator\footnote{The operator $\mathcal L$ is the Ornstein-Uhlenbeek operator $D^* D$, where $D^*$ is the adjoint of $D$, for which we need to be precise on functional spaces -- this is not important for the estimates} with bounded norm in the sense that for all $X$ we have
\begin{equation*}
\expec{\|(1+\mathcal L)^{-1} DX\|_{\beta}^2} \,\le\, \expec{\|  DX\|_{\beta}^2} .
\end{equation*}
Here again, what is important is that the formula (i.e. the identity) exists, and that the abstract operator $(1+\mathcal L)^{-1}$ can be removed in actual estimates.

\subsection{Control of the distance to a normal random variable}

A chaos expansion for a random variable is like a Taylor expansion for a function.
At first order, we compare a random variable to a constant.
At second order, we compare a random variable to a normal random variable (that is, with a Gaussian distribution), etc.
When the law of a sequence of random variables converges to a normal law (which means that the fluctuations tend to become Gaussian), we say that the sequence of random variables is asymptotically normal.
Like weak convergence, convergence in law is metrizable and 
the second Malliavin derivative can be used to control the distance of the law of a random variable to normality. This is the so-called Stein-Malliavin calculus.
We need a norm on second-order Malliavin derivative, and we set for all $f:\R\times \R \to \R$,
\begin{equation}
\3 f \3_\beta := \begin{cases}
 \sup \Big\{ \Big|\iint_{\R \times \R} \zeta(x)\zeta(y) f(x,y)dxdy\Big|\,:\,   \|\zeta\|_{L^{\frac 2\beta,\infty}(\R)}=1\Big\} ,&\text{if $ \beta < 1$};\\
\sup  \Big\{ \Big|\iint_{\R \times \R} w_c(x) \zeta(x) w_c(y)\zeta(y) f(x,y)dxdy\Big|\,:\,  \|\zeta\|_{L^{2,\infty}(\R)}=1  \Big\} ,&\text{if $\beta = 1$};\\
\sup  \Big\{ \Big|\iint_{\R \times \R}    \zeta(x)  \zeta(y) f(x,y)dxdy\Big|\,:\,  \|\zeta\|_{L^{2,\infty}(\R)}=1  \Big\} ,&\text{if $\beta > 1$},
\end{cases}
\end{equation}
where $w_c:x \mapsto \log(2+|x|)^\frac12$, and where we have used the short-hand notation
$$
\|\zeta\|_{L^{p,\infty}(\R)}:=\Big(\int_\R \big(\sup_{B(x)} |\zeta|\big)^p dx\Big)^\frac1p.
$$
We then have the following second-order Poincar\'e-Malliavin inequality
\begin{equation}\label{e.SndPI}
\dTV{ \frac{X-\expec{X}}{\sigma}}{\Nc}
\,\lesssim\,\sigma^{-2}\,\expec{\3D^2X\3_{\beta}^4}^\frac14\expec{\|DX\|_\beta^4}^\frac14,
\end{equation}
where $\sigma^2:=\var{X}$ is assumed to be non-zero,
and $\dTV{}{}$ denotes the total variation distance between random variables, which is defined 
for any random variables $\mathcal M_1$ and $\mathcal M_2$ via
$$
\dTV{\mathcal M_1}{\mathcal M_2}=\sup_{A \subset \R} |\mu_1(A)-\mu_2(A)|,
$$
where for $i=1,2$, $\mu_i$  is the distribution of $\mathcal M_i$  
(that is, for any Borel set $A \subset \R$, $\mu_i(A) := \expec{\mathds 1_{\mathcal M_i \in A}}$).

\section{Qualitative homogenization via explicit formulas}\label{sec:qual}

\subsection{Random coefficients}

We consider the popular log-normal coefficient fields, which we presently introduce.
The log-normal coefficient field $a:\R \to (0,+\infty)$ is given by 
$x \mapsto a(x)=\exp(G(x))$, where $G$ is a Gaussian field as the one considered in Section~\ref{sec:Malliavin}.
Since $G$ takes values in $\R$, $a$ is neither bounded nor bounded away from zero.
However, by the  Kolmogorov-Chentsov theorem, since $\mathcal C$ is assumed to be H\"older-continuity at zero, $a$ is (H\"older-)continuous for almost every realization of $G$, and therefore positive and finite on $(0,1)$ (although $\inf_{[0,1]}a>0$ and $\sup_{[0,1]} a<\infty$ depend on the realization). Hence \eqref{0:1D-1} is well-posed almost surely by the Lax-Milgram theorem.

\medskip

We now estimate moments of $a$ and of $a^{-1}$, which will be useful in the sequel.
\begin{lem}\label{eq:mom-a}
For all $p\in \N$ we have
$\expec{a^p}^\frac1p\,=\,\expec{a^{-p}}^\frac1p \,=\, \exp(\frac{\calC(0)}{2}p)$.
\end{lem}
Indeed, since $G$ and $-G$ have the same law, $\expec{a^{-p}}=\expec{\exp(-pG)}=\expec{\exp(pG)}$, and a direct calculation yields by reconstructing a square in the exponential
\begin{equation}\label{e.calc-exp}
\expec{\exp(pG)} = \int_{-\infty}^\infty \exp(pt) \frac{1}{\sqrt{2\pi\calC(0)}}\exp(-\frac{t^2}{2\calC(0)})dt  =  \exp(\tfrac{\calC(0)}{2}p^2).
\end{equation}

\subsection{Qualitative homogenization}

Fix a realization of $G$.
By direct integration the unique solution of \eqref{0:1D-1} is explicitly given by
\begin{equation*}  
x\mapsto \int_0^x \frac1{a(\frac y\e)}f(y)dy+\int_0^x \frac{C_1^\e}{a(\frac y\e)}dy+C_2^\e,
\end{equation*}
where $C_1^\e$ and $C_2^\e$ are integration constants.
The condition $u_\e(0)=0$ yields $C_2^\e=0$ whereas the condition $u_\e(1)=0$ yields 
\begin{equation}\label{0:1defC1}
C_1^\e=-\bigg(\int_0^1 \frac1{a(\frac y\e)}dy\bigg)^{-1} \int_0^1 \frac1{a(\frac y\e)}f(y)dy,
\end{equation}
to the effect that
\begin{equation}\label{0:1D-2} 
u_\e:x\mapsto \int_0^x \frac1{a(\frac y\e)}f(y)dy-\int_0^x \frac{1}{a(\frac y\e)}dy\bigg(\int_0^1 \frac1{a(\frac y\e)}dy\bigg)^{-1} \int_0^1 \frac1{a(\frac y\e)}f(y)dy.
\end{equation}
With this explicit formula at hands we may investigate the asymptotic behavior of $u_\e$.
We start with a qualitative convergence result based on the ergodic theorem of Subsection~\ref{sec:erg-th}.
\begin{theor}\label{0:1D-thm-qual}
The solution $u_\e$ of \eqref{0:1D-1} given by \eqref{0:1D-2}
converges almost surely pointwise to the solution 
\begin{equation}\label{0:1D-3} 
\bar u:x\mapsto \int_0^x \frac1{\bar a}f(y)dy-x \int_0^1  \frac1{\bar a}f(y)dy
\end{equation}
of the effective equation 
\begin{equation}\label{0:1D-4}
(\bar a \bar u'(x))'=f'(x), \quad \bar u(0)=0, \quad  \bar u(1)=0,
\end{equation}
where the effective coefficient is given by
\begin{equation}\label{0:1D-5}
\bar a = \exp(-\tfrac{\calC(0)}{2})>0.
\end{equation}
\end{theor}
This result shows in particular that the family of solutions \eqref{0:1D-2} of  \eqref{0:1D-1} converges (pointwise in space, almost surely in probability) 
to the solution of a similar equation \eqref{0:1D-4}, this time with constant coefficient $\bar a$ (which only depends on the law of $a$ via \eqref{0:1D-5} --- this 
explicit form is however specific to dimension $d=1$ \emph{and} to log-normal coefficients),
and with the same forcing term $f$ and initial conditions.
\begin{proof}[Proof of Theorem~\ref{0:1D-thm-qual}]
Applying the ergodic theorem to the random field $\frac1 a$ with the test functions $f$ and $1$ yields for all $x\in \R^d$  almost surely
$$
\lim_{\e \downarrow 0} \int_0^x \frac1{a(\frac y\e)} f(y)dy =\expec{\frac1a} \int_0^x  f(y)dy, \quad \lim_{\e \downarrow 0} \int_0^x \frac1{a(\frac y\e)} dy =\expec{\frac1a} x,
$$
so that $u_\e(x)\to \bar u(x)$ almost surely (strictly speaking, one should first consider a countable set of points $x$, say $\mathbb Q\cap [0,1]$, for which the convergence
holds almost surely since a countable intersection of events of full probability remains of full probability, and conclude by uniform continuity of $u_\e$ and $\bar u$).
We conclude by noticing that $\bar u$ is indeed the unique solution of~\eqref{0:1D-4}
with $\bar a :=\expec{\frac1a}^{-1}$, and \eqref{0:1D-5} follows by Lemma~\ref{eq:mom-a} for $p=-1$.
\end{proof}
The map $\bar u$ obviously does not depend on $\e$, whereas the map $u_\e$  
displays oscillations at scale $\e$ (via the rescaling $a(\frac \cdot \e)$).
Hence, although $u_\e$ and $\bar u$ are close pointwise, their gradients cannot be close
pointwise, as confirmed by the following identities:
\begin{eqnarray}
u_\e'(x)&=&\frac{1}{a(\frac x\e)}f(x)-\frac{1}{a(\frac x\e)}\bigg(\int_0^1 \frac1{a(\frac y\e)}dy\bigg)^{-1} \int_0^1 \frac1{a(\frac y\e)}f(y)dy, 
\label{0:1D-16}
\\
\bar u'(x)&=&\expec{\frac 1a} f(x)-\expec{\frac 1a}\int_0^1 f(y)dy.\nonumber
\end{eqnarray}
In view of these formulas, on may ``reconstruct'' $u_\e'$ using $\bar u'$ and quantities that only depend on $a$ (and in particular, not
on $f$ nor on boundary conditions) up to an error which vanishes in the limit $\e \downarrow 0$ almost surely.
The following bears the name of a corrector result: One corrects  $\bar u'$ to get a pointwise accurate approximation
of $u_\e'$ (by reconstructing a posteriori the oscillations)\footnote{In this interpretation, it is more instructive to write
the LHS of \eqref{0:1D-6} as $u_\e'(x)-\bar u'(x)-\bar u'(x)\bar a (\frac1{a(\frac x\e)}-\frac1{\bar a})$, where $\bar a (\frac1{a(\frac x\e)}-\frac1{\bar a})$ is the
classical corrector gradient --- see Section~\ref{0.quant-revisited}.}.
\begin{theor}\label{0:1D-thm-qual2}
The solutions $u_\e$ of \eqref{0:1D-1} and $\bar u$ of \eqref{0:1D-4} satisfy
almost surely for all $x\in [0,1]$
\begin{eqnarray}\label{0:1D-6}
\lim_{\e \downarrow 0} u_\e'(x)-\bar u'(x) \frac{\bar a}{a(\frac x\e)}&=&0,
\end{eqnarray}
where $\bar a$ is given by \eqref{0:1D-5}.
\end{theor}
\begin{proof}[Proof of Theorem~\ref{0:1D-thm-qual2}]
A direct calculation yields
\begin{equation}\label{0:1D-13}
u_\e'(x)-\bar u'(x)\frac{\bar a}{a(\frac x\e)}
\,=\,\frac1{a(\frac x\e)} \bigg(\int_0^1f(y)dy- \Big(\int_0^1 \frac{1}{a(\frac y\e)}dy\Big)^{-1}
\int_0^1 \frac1{a(\frac y\e)}f(y)dy \Bigg),
\end{equation}
so that the claim follows again from the ergodic theorem applied to the random field $\frac 1a$ and the test functions $1$ and $f$.
\end{proof}
Theorems~\ref{0:1D-thm-qual} and~\ref{0:1D-thm-qual2} are both qualitative in the sense that they do not come with error estimates.
In view of the proofs, in order to make these results quantitative, one needs to quantify the convergence in the ergodic theorem, and therefore appeal (in view of the Gaussianity of $G$) to Malliavin calculus.

\section{Quantification of oscillations}\label{sec:2s}

The following result quantifies (optimally) Theorems~\ref{0:1D-thm-qual} and~\ref{0:1D-thm-qual2}.
\begin{theor}\label{0:1D-thm-quant1}
Let the covariance $\calC$ of $G$ satisfy \eqref{e.decay-Co} for some $\beta>0$.
Then the solutions $u_\e$ of \eqref{0:1D-1} and $\bar u$ of \eqref{0:1D-4} satisfy
 for all $x\in [0,1]$
\begin{equation}
\expec{|u_\e(x)-\bar u(x)|^2}^\frac12 + \expec{\Big|u_\e'(x)-\bar u'(x)\frac{\bar a}{a(\frac x\e)}\Big|^2}^\frac12   \,
\lesssim\, \ \|f\|_{L^2(0,1)}  \pi_\beta(\e),
\label{0:1D-7}
\end{equation}
where 
\begin{equation}\label{e.pibeta}
 \pi_\beta(\e):=
 \begin{cases}
\e^{\frac{\beta}{2}},&\text{if $0<\beta<1$},\\
\sqrt \e |\log(\e)|^\frac12,&\text{if $\beta=1$},\\
\sqrt \e ,&\text{if $\beta>1$}.
\end{cases}
\end{equation}
\end{theor}
Theorem~\ref{0:1D-thm-quant1} gives a very accurate description of the oscillations of $u_\e$ and of $u'_\e$.
Within our assumptions, estimates~\eqref{0:1D-7} are optimal in terms of scaling, that is, in terms of powers of $\e$.
The norm in probability (that is, the second moment) can however be largely improved.
The proof below actually shows that these estimates hold 
in the form
\[
\expec{\frac1C \log^2\Big(1+\frac{|u_\e(x)-\bar u(x)|+|u_\e'(x)-\bar u'(x)\frac{\bar a}{a(\frac x\e)}|}{\|f\|_{L^2(0,1)}  \pi_\beta(\e)}\Big)} \le 2
\] 
for some $0<C<\infty$ independent of $\e$,
see~\eqref{0:1D-14}. In more analytical terms, the random variable $\frac{|u_\e(x)-\bar u(x)|+|u_\e'(x)-\bar u'(x)\frac{\bar a}{a(\frac x\e)}|}{\|f\|_{L^2(0,1)}  \pi_\beta(\e)}$ is not bounded but it belongs (uniformly) to some Orlicz space in probability (that contains all $L^p$ spaces).

\medskip

For future reference, we single out an elementary result we shall use in a simpler version in the proof  of Theorem~\ref{0:1D-thm-quant1}.
\begin{lem}\label{lem:Minko}
For all $\alpha_1,\alpha_2,\gamma \in \R$, $p\ge q\ge 1$, deterministic $f \in L^q(\R)$, and $t\in \R$, we have
\begin{multline}\label{e.Minko}
\expec{\Big(\int_0^t a(z)^{\alpha_1q} \big(\fint_0^t \tfrac1a\big)^{\alpha_2 q} |f(z)|^q w_c(z)^{\gamma q}dz\Big)^\frac pq}^\frac1p
\\
\, \le \exp(\tfrac12 (|\alpha_1|+|\alpha_2|)^2\calC(0)p) \|fw_c^\gamma\|_{L^q(0,t)}.
\end{multline}
\end{lem}
\begin{proof}[Proof of Lemma~\ref{lem:Minko}]
By the Minkowski inequality, since $p\ge q$ we have
\begin{multline}\label{e.Minko+}
\expec{\Big(\int_0^t a(z)^{\alpha_1q} \big(\fint_0^t \tfrac1a\big)^{\alpha_2 q} |f(z)|^q \omega_c(z)^{\gamma q}dz\Big)^\frac pq}^\frac1p
\\
\le \, \Big(\int_0^t \expec{a(z)^{\alpha_1p} \big(\fint_0^t \tfrac1a\big)^{\alpha_2 p}}^\frac qp |f(z)|^q w_c^{\gamma q}(z)dz\Big)^\frac1q.
\end{multline}
It remains to treat the expectation, and we assume without loss of generality that $\alpha_1,\alpha_2\ge 0$ (otherwise we replace $a$ by $\tfrac1a$ or $\fint_0^t \frac1a$ by $\big(\fint_0^t \tfrac1a)^{-1}$). By H\"older's inequality with exponents $(\frac{\alpha_1+\alpha_2}{\alpha_1},\frac {\alpha_1+\alpha_2}{\alpha_2})$ we have
\begin{equation*}
 \expec{a(z)^{\alpha_1p} \big(\fint_0^t \tfrac1a\big)^{\alpha_2 p}}
 \\\le\, \expec{a^{p(\alpha_1+\alpha_2)}}^\frac{\alpha_1}{\alpha_1+\alpha_2}
  \expec{\big(\fint_0^t\tfrac1a\big)^{p(\alpha_1+\alpha_2)}}^\frac{\alpha_2}{\alpha_1+\alpha_2},
\end{equation*}
so that by Jensen's inequality, stationarity, and using that $a(x)=\exp(G(x))$, we obtain 
using Lemma~\ref{eq:mom-a}
$$
 \expec{a(z)^{\alpha_1p} \big(\fint_0^t \tfrac1a\big)^{\alpha_2 p}}\,\le \, \expec{\exp(p(\alpha_1+\alpha_2)G)} \le  \exp(\frac{1}2 \calC(0)p^2(\alpha_1+\alpha_2)^2). 
 $$
This yields~\eqref{e.Minko} in combination with~\eqref{e.Minko+}. 
\end{proof}
We conclude with the proof of Theorem~\ref{0:1D-thm-quant1}.
\begin{proof}[Proof of Theorem~\ref{0:1D-thm-quant1}]
The proof is based on the moment bounds \eqref{e.Mal:moment}
 for the Malliavin derivative.
Both LHS terms of  \eqref{0:1D-7} have the same scaling and can be treated similarly.
We only display the argument for the first term and set  $w_\e=u_\e-\bar u$.
We then consider the splitting $w_\e=w_{\e,1}-w_{\e,2}-w_{\e,3}$,
where
\begin{eqnarray*}
w_{\e,1}(x)&:=&\int_0^x (\frac1{a(\frac y\e)}-\frac1{\bar a})f(y)dy-x \int_0^1  (\frac1{a(\frac y\e)}-\frac1{\bar a})f(y)dy,\\
w_{\e,2}(x)&:=&\int_0^x (\frac{1}{a(\frac y\e)}-\frac1{\bar a})dy\bigg(\int_0^1 \frac1{a(\frac y\e)}dy\bigg)^{-1} \int_0^1 \frac1{a(\frac y\e)}f(y)dy,\\
w_{\e,3}(x)&:=&x \frac1{\bar a}\bigg( \Big(\int_0^1 \frac1{a(\frac y\e)}dy\Big)^{-1}-\bar a\bigg) \int_0^1 \frac1{a(\frac y\e)}f(y)dy.
\end{eqnarray*}
For all $x\in [0,1]$ and $p\ge 1$, by the triangle inequality in $L^p(\Omega)$
$$
\expec{|w_\e(x)|^p}^\frac1p \,\le \, \expec{|w_{\e,1}(x)|^p}^\frac1p +\expec{|w_{\e,2}(x)|^p}^\frac1p +\expec{|w_{\e,2}(x)|^p}^\frac1p.
$$
We then  change variables to make heterogeneities of order $1$, that is $y \mapsto y/\e$.
For the first RHS term we directly have
$$
\expec{|w_{\e,1}(x)|^p}^\frac1p \,\le \, x\expec{\Big|\fint_0^{\frac x\e} (\frac1{a(y)}-\frac1{\bar a})f(\e y)dy\Big|^p}^\frac1p+x\expec{\Big|\fint_0^{\frac 1\e} (\frac1{a( y)}-\frac1{\bar a})f(\e y)dy\Big|^p}^\frac1p.
$$
For the second RHS term we separate the first factor in $w_{\e,2}$ (which will give the smallness) from the other two factors, and use that $a(x)=\exp(G(x))$ so that by H\"older's inequality with exponents $(2p,4p,4p)$ and a multiple use of Jensen's inequality we obtain
$$
\expec{|w_{\e,2}(x)|^p}^\frac1p \,\le \,x \expec{\Big|\fint_0^{\frac x\e} (\frac1{a( y)}-\frac1{\bar a})dy\Big|^{2p}}^\frac1{2p} \expec{\exp(4pG)}^\frac1{2p}\|f\|_{L^2(0,1)}.$$
For the last RHS term we proceed similarly, but first use the reformulation 
\begin{equation}\label{0:1D-transfo-inverse}
\Big(\int_0^1 \frac1{a(\frac y\e)}dy\Big)^{-1}-\bar a = \Big(\frac1{\bar a}-\fint_0^1 \frac1{a(\frac y \e)}dy\Big) \Big( \frac1{\bar a} \fint_0^1 \frac1{a(\frac y\e)}dy\Big)^{-1},
\end{equation}
so that as for the second term we obtain by H\"older's and Jensen's inequalities
$$
\expec{|w_{\e,3}(x)|^p}^\frac1p \,\le \,x \expec{\Big|\fint_0^{1} (\frac1{a( y)}-\frac1{\bar a})dy\Big|^{2p}}^\frac1{2p} \expec{\exp(4pG)}^\frac1{2p}\|f\|_{L^2(0,1)}.
$$
Hence,
\begin{eqnarray}
\expec{|w_\e(x)|^p}^\frac1p 
&\le& \e \expec{\Big|\int_0^{\frac x\e} (\frac1{a(y)}-\frac1{\bar a})f(\e y)dy\Big|^p}^\frac1p+\e x \expec{\Big|\int_0^{\frac 1\e} (\frac1{a( y)}-\frac1{\bar a})f(\e y)dy\Big|^p}^\frac1p
\nonumber\\
&&+\e \expec{\exp(4pG)}^\frac1{2p}\|f\|_{L^2(0,1)}\expec{\Big|\int_0^{\frac x\e} (\frac1{a( y)}-\frac1{\bar a})dy\Big|^{2p}}^\frac1{2p}
\nonumber\\
&&+\e x\expec{\exp(4pG)}^\frac1{2p}\|f\|_{L^2(0,1)}\expec{\Big| \int_0^{\frac1\e} (\frac1{a(y)}-\frac1{\bar a})dy\Big|^{2p} }^\frac1{2p}.\label{0:1D-9}
\end{eqnarray}
In order to appeal to the moment bound \eqref{e.Mal:moment}, it remains
to  
calculate the Malliavin derivative \eqref{e.Mallder} of random variables of the form $Y=\int_0^t (\frac1{a(y)}-\frac1{\bar a})g(y)dy=\int_0^t (\exp(-G(y))-\frac1{\bar a})g(y)dy$ for fixed test function $g$
and real number $t$, which takes the form
\begin{equation}\label{0:1D-17}
D_z Y=-\exp(-G(z))g(z) \mathds{1}_{z\in (0,t)}.
\end{equation}
By the definition \eqref{e.Mallnorm} of the norm $\|\cdot\|_\beta$ and Lemma~\ref{lem:Minko}, 
$$
\expec{\|DY\|_\beta^{2p}}^\frac1{2p} \,\le \, \exp( \calC(0)p) \|g \mathds 1_{(0,t)}\|_\beta,
$$
so that by  \eqref{e.Mal:moment}
\begin{equation}\label{e.2scquant-1}
\expec{|Y|^{2p}}^\frac1{2p} \,\lesssim\, \sqrt p \exp(\calC(0) p) \|g\mathds 1_{(0,t)}\|_{\beta}.
\end{equation}
We apply~\eqref{e.2scquant-1} to $g \equiv 1$ and $g\equiv f(\e \cdot)$ with $t=\frac x\e$ and $t=\frac1\e$, in which case we have (since $0\le x\le 1$ and  $\omega_c(\frac y \e) \le \omega_c(\e^{-1}) \omega_c(y)$)
$$
 \|g\omega_c ^{ \gamma}\mathds 1_{(0,t)}\|_{L^q(\R)} \,\le \, \|g\|_{L^2(0,1)}\omega_c^\gamma(\tfrac1\e) \e^{-\frac1q}.
$$
Hence \eqref{0:1D-9} turns into the improved form of \eqref{0:1D-7}
\begin{equation} \label{0:1D-14}
\expec{|w_\e(x)|^p}^\frac1p \, \lesssim \, \sqrt p \exp(5\calC(0)p) \|f\|_{L^2(0,1)}  
\pi_\beta(\e).
\end{equation}
\end{proof}

\section{Quantification of fluctuations: Scaling law of observables}\label{sec:law}

\subsection{Scaling of fluctuations}

As emphasized in the introduction, $u_\e$ displays both spatial oscillations (as characterized by Theorem~\ref{0:1D-thm-quant1}) and random fluctuations ($u_\e$ is a random field). 
We now turn to the latter, and focus on observables of 
the solution $u_\e$.
An observable is typically a local average of $u_\e$ or of its gradient -- it is another name for a measurement in a physical experiment.
In the first part of this article, we have shown how to replace $u_\e$ by some deterministic function $\bar u$ up to some precision. Here we aim at giving an asymptotic description of the law of the fluctuations of observables of $u_\e$.
This can be seen as an uncertainty quantification: we characterize the law of the observables given uncertain coefficients.
More precisely, given $g \in  C^1([0,1])$, we consider the obserbable 
$$
I_\e(f,g) \,=\, -\int_0^1 u_\e(x) g'(x)dx=\int_0^1 u_\e'(x)g(x)dx
$$
(recall that $f$ is in the RHS of \eqref{0:1D-1}).
A similar argument as in the proof of Theorem~\ref{0:1D-thm-quant1} allows to 
characterize the order of magnitude of $\var{I_\e(f,g)}$ wrt $\e$. 
\begin{theor}\label{0:1D-thm-quant2}
Let the covariance $\calC$ of $G$ satisfy \eqref{e.decay-Co} for some $\beta>0$, $f,g \in C^1([0,1])$, and let $I_\e(f,g)=\int_0^1 u_\e'(x)g(x)dx$, where $u_\e$ is the solution of  \eqref{0:1D-1}.
We have
\begin{eqnarray}\label{0:1D-15}
\var{I_\e(f,g)}&\lesssim &  \|f\|_{L^4(0,1)}^2 \|g\|_{L^4(0,1)}^2\pi_\beta(\e)^2
\\
&=&\|f\|_{L^4(0,1)}^2 \|g\|_{L^4(0,1)}^2
\begin{cases}
\e^{\beta},&\text{if $0<\beta<1$},\\
 \e |\log(\e)| ,&\text{if $\beta=1$},\\
 \e ,&\text{if $\beta>1$}.
\end{cases}
\nonumber
\end{eqnarray}
\end{theor}
Estimate~\eqref{0:1D-15} encodes the natural scaling of fluctuations
associated with a Gaussian field with covariance $\calC$ (in dimension $d\ge 1$, one would have the CLT scaling $\e^d$ for $\beta>d$).
\begin{proof}[Proof of Theorem~\ref{0:1D-thm-quant2}]
By \eqref{0:1D-16},
\begin{equation} \label{0:1D-19}
I_\e(f,g)\,=\,\fint_0^{\frac 1\e} \frac{1}{a(x)}f(\e x)g(\e x)dx-\bigg(\fint_0^{\frac1\e} \frac1{a( y)}dy\bigg)^{-1}  \fint_0^{\frac 1\e}\frac{1}{a(y)}g(\e y)dy\fint_0^{\frac 1\e} \frac1{a(y)}f(\e y)dy.
\end{equation}
We shall appeal again to the variance estimate for the Malliavin derivative $\var{I_\e(f,g)}\,\lesssim\,   \expec{\|DI_\e(f,g)\|_\beta^2}$, and
need to identify the latter.
Using \eqref{0:1D-17}, the Leibniz rule for the Malliavin derivative, 
and the identity $D_z\frac1a(x)=-\exp(-G(z)) \delta(x-z)=-\frac1a(z)\delta(x-z)$, we obtain for all $z\in \R$ 
\begin{eqnarray}  
\lefteqn{D_z I_\e(f,g)=- \e \frac1{a(z)} f(\e z)g(\e z) \mathds{1}_{z\in (0,\frac1\e)} }\nonumber
\\
&&+\e\frac1{a(z)}f(\e z)  \mathds{1}_{z\in (0,\frac1\e)} \bigg(\fint_0^{\frac1\e} \frac1{a( y)}dy\bigg)^{-1}  \fint_0^{\frac 1\e}\frac{1}{a(y)}g(\e y)dy \nonumber
\\
&&+\e\frac1{a(z)}g(\e z)  \mathds{1}_{z\in (0,\frac1\e)} \bigg(\fint_0^{\frac1\e} \frac1{a( y)}dy\bigg)^{-1} \fint_0^{\frac 1\e} \frac1{a(y)}f(\e y)dy\nonumber
\\
&&-\e \frac1{a(z)}\mathds{1}_{z\in (0,\frac1\e)}\bigg(\fint_0^{\frac1\e} \frac1{a( y)}dy\bigg)^{-2}  \fint_0^{\frac 1\e}\frac{1}{a(y)}g(\e y)dy\fint_0^{\frac 1\e} \frac1{a(y)}f(\e y)dy. \label{0:1D-18}
\end{eqnarray}
We then conclude by a suitable use of Lemma~\ref{lem:Minko}.
\end{proof}

\medskip

When we go beyond scalings and identify limiting laws, results depend on whether the covariance function $\calC$ is integrable or not, that is, whether $\beta>1$ or $\beta \le 1$.
To ease the reading, we distinguish these cases both in terms of statements and proofs, and first give the proof for $\beta>1$. The non-integrable case $\beta \le 1$ is presented afterwards and capitalizes on the proof for $\beta>1$.

\medskip

There are two main important differences between $\beta>1$ and $\beta<1$ ($\beta=1$ is in between).
For $\beta>1$, there is a unique limiting law, and it is a white noise (the covariance is trivial), cf.~Theorem~\ref{0:1D-thm-quant3}.
For $\beta<1$, convergence in law holds up to extracting a subsequence (or assuming some additional scaling property as we do here), and the limiting law is a fractional noise (the covariance has a nontrivial kernel), cf.~Theorem~\ref{0:1D-thm-quant3-non-int}.

\subsection{Scaling law for integrable correlations}

In this paragraph we place ourselves in the case of integrable covariance function, that is, for $\beta>1$.
The following result goes further in the analysis and establishes that, when properly rescaled, $I_\e(f,g)$
satisfies a (quantitative) central limit theorem.
We split the statement into two parts: the convergence of the covariance structure and 
the normal approximation.
We start with the covariance structure.
\begin{prop}\label{prop:cov-int}
Let $\beta>1$, and $G \not\equiv 0$ (otherwise $a \equiv 1$ is constant).
Let $f,g \in C^1([0,1])$, let $u_\e$ be the solution of  \eqref{0:1D-1},  $v_\e$
be the solution of  \eqref{0:1D-1} for $f$ replaced by $g$, and set $I_\e(f,g):=\int_0^1 u_\e'(x)g(x)dx=-\int_0^1 v_\e'(x)f(x)dx$ and  $\sigma_\e(f,g):=\sqrt{\e^{-1}\var{I_\e(f,g)}}$.
Then we have 
\begin{equation}\label{e.conv-cov-int}
|\sigma_\e(f,g)^2-\sigma(f,g)^2|\,\lesssim\, \e^{\frac 12 \wedge (\beta-1)},
\end{equation}
where 
\begin{equation}\label{0:1D-lim-var}
\sigma^2(f,g):= \calQ  \int_0^1 \Big(f-\int_0^1f\Big)^2\Big(g-\int_0^1 g\Big)^2 ,
\end{equation}
with $\calQ=\int_\R \cov{\frac1a(x)}{\frac1a(0)}dx>0$.
\end{prop}
\begin{rem}
The property $\calQ>0$ is a consequence of our choice $a(x)=\exp(G(x))$.
We refer the reader to \cite{Taqqu,BGMP-08,LNZH-17} for degenerate cases for which
$\int_\R \cov{\frac1a(x)}{\frac1a(0)}dx=0$ and we have to go further to find the next nonzero term in the chaos expansion (one speaks of Hermite rank).
The property $\calQ>0$ is however generic (in the sense that suitable small perturbations of $\calC$ make the problem non-degenerate).
\end{rem}
We now turn to the normal approximation.
\begin{prop}\label{prop:NA-int}
Let $\beta>1$.
Let $f,g \in C^1([0,1])$, let $u_\e$ be the solution of  \eqref{0:1D-1},  $v_\e$
be the solution of  \eqref{0:1D-1} for $f$ replaced by $g$, and set $I_\e(f,g):=\int_0^1 u_\e'(x)g(x)dx=-\int_0^1 v_\e'(x)f(x)dx$ and $\sigma_\e(f,g):=\sqrt{\e^{-1}\var{I_\e(f,g)}}$.
Then, as long as $\sigma_\e(f,g)\ne 0$, we have
\begin{equation}\label{e:NA-int}
d_{TV}\Big(\frac{I_\e(f,g)-\expec{I_\e(f,g)}}{\sqrt\e \sigma_\e(f,g)},\mathcal N \Big) \,\lesssim_{f,g}\, \frac1{\sigma_\e(f,g)^2}\sqrt \e \exp(C |\log \e|^\frac12),
\end{equation}
where $\mathcal N$ is a centered normal variable of variance unity,
and $d_{TV}$ denotes the total variation distance between probability measures.
\end{prop}
\begin{rem}
The sub-algebraic correction $\exp(C |\log \e|^\frac12)$ (it is asymptotically smaller than any negative power of $\e$) is due to the (relatively) poor stochastic integrability of log-normal fields. If $\frac1a$ was bounded from above and isolated from zero, there would be no such correction. 
\end{rem}
These two propositions combine into the following quantitative central limit theorem.
\begin{theor}\label{0:1D-thm-quant3}
Let $\beta>1$, and $G \not\equiv 0$.
Let $f,g \in C^1([0,1])$, let $u_\e$ be the solution of  \eqref{0:1D-1},  $v_\e$
be the solution of  \eqref{0:1D-1} for $f$ replaced by $g$,  set $I_\e(f,g):=\int_0^1 u_\e'(x)g(x)dx=-\int_0^1 v_\e'(x)f(x)dx$ and $\sigma_\e(f,g):=\e^{-1}\var{I_\e(f,g)}$, 
and let $\sigma(f,g)$ be given by \eqref{0:1D-lim-var}.
If both $f$ and $g$ are not constant (otherwise $u_\e\equiv 0$ or $v_\e\equiv 0$ so that $I_\e(f,g)\equiv 0$), then $\sigma(f,g) >0$ and 
$I_\e(f,g)$ satisfies
\begin{equation}\label{0:1D-20}
d_{TV}\Big(\frac{I_\e(f,g)-\expec{I_\e(f,g)}}{\sqrt\e \sigma(f,g)},\mathcal N \Big) \,\lesssim_{f,g}\, \sqrt \e \exp(C|\log \e|^\frac12)+ \e^{\beta-1}.
\end{equation}
\end{theor}

\subsubsection{Covariance structure: proof of Proposition~\ref{prop:cov-int}}

By Theorem~\ref{0:1D-thm-quant2}  the empirical variance
$
\sigma_\e^2(f,g):=\e^{-1}\var{I_\e(f,g)},
$
is of order 1.
All the terms in $I_\e(f,g)$ are integrals of $\frac1a$ with a test function, except the term $\Big(\fint_0^\frac1\e\frac1a\Big)^{-1}$ (it is the inverse
of an average).
In the first  step, we show that this term can be rewritten as an average of $\frac1a$ up to a small error.
In the second  step, we then argue that $\sigma_\e^2(f,g)$ can be 
approximated by the variance of a sum of four spatial averages of $\frac1a$, which we then expand as the sum of 16 terms.
Each of the latter is a covariance of spatial averages of $\frac1a$, which we then analyze.
We then prove~\eqref{e.conv-cov-int}, and conclude with the non-degeneracy.

\medskip

\step{1} Proof of the identity
\begin{equation}\label{0:1D-rewriting-var0}
\frac1{\fint_0^{\frac1\e} \frac1a }\,=\,\frac{2\expec{\frac1a}-\fint_0^{\frac1\e} \frac1a}{\expec{\frac1a}^2}+\mu_\e, \quad \mu_\e\,:=\,\frac{(\expec{\frac1a}-\fint_0^{\frac1\e} \frac1a)^2}{\expec{\frac1a}^2\fint_0^\frac1\e \frac1a}
\end{equation}
where 
$\mu_\e$
satisfies
\begin{equation}\label{0:1D-rewriting-var1}
\expec{\mu_\e^2}\lesssim \e^2.
\end{equation}
Before we prove \eqref{0:1D-rewriting-var0} and  \eqref{0:1D-rewriting-var1}, let us motivate the form of \eqref{0:1D-rewriting-var0}.
We are interested in variances.
Since $\fint_0^{\frac1\e} \frac1a$ is close to $\expec{\frac1a}$ and its variance is of order $\e$,
we have $\fint_0^{\frac1\e} \frac1a=\expec{\frac1a}+b_\e$ for some $b_\e$ with $\expec{b_\e^2}\lesssim \e$, and \eqref{0:1D-rewriting-var0} is
nothing but the Taylor-expansion $\Big(\expec{\frac1a}+b_\e\Big)^{-1}=- b_\e \expec{\frac{1}{a}}^{-2}+O(\e)$, where the remainder (and the norm
in which it is small --- the second stochastic moment) is made precise.

\medskip

The identity~\eqref{0:1D-rewriting-var0} is similar to \eqref{0:1D-transfo-inverse}:
$$
\frac1{\fint_0^{\frac1\e} \frac1a }-\frac1{\expec{\frac1a}}= \frac{\expec{\frac1a}-\fint_0^{\frac1\e} \frac1a}{\expec{\frac1a}\fint_0^{\frac1\e} \frac1a} 
\,=\,  \frac{\expec{\frac1a}-\fint_0^{\frac1\e} \frac1a}{\expec{\frac1a}^2} - \frac{\expec{\frac1a}-\fint_0^{\frac1\e} \frac1a}{\expec{\frac1a}}\Big(\frac1{\expec{\frac1a}}-\frac1{\fint_0^{\frac1\e} \frac1a} \Big).
$$
We have already encountered \eqref{0:1D-rewriting-var1}, which is a reformulation of \eqref{e.2scquant-1} using Lemma~\ref{eq:mom-a} and Cauchy-Scwharz' inequality.

\medskip

\step{2} Decomposition of $\sigma_\e^2(f,g)$ and proof of 
\begin{equation}\label{0:1D-thm5-1.2}
|\sigma_\e^2(f,g)-\var{X_{1,\e}+X_{2,\e}-X_{3,\e}-X_{4,\e}}| \,\lesssim \, \sqrt\e,
\end{equation}
where 
\begin{equation*}
\begin{array}{ll}
X_{1,\e}\,:=\,\e^{-\frac12} \fint_0^{\frac1\e} \frac1{a(x)}f(\e x)g(\e x)dx, & X_{2,\e}\,:=\,\e^{-\frac12}\Big( \fint_0^{\frac1\e} \frac1{a} \Big)\int_0^1f\int_0^1 g,\\
X_{3,\e}\,:=\,\e^{-\frac12} \Big(\fint_0^{\frac1\e} \frac1{a(x)}g(\e x)dx\Big)\int_0^1f, & X_{4,\e}\,:=\,\e^{-\frac12}\Big( \fint_0^{\frac1\e} \frac1{a(x)}f(\e x)dx \Big)\int_0^1  g.
\end{array}
\end{equation*}
Starting point is the following observation for five bounded random variables $\{\alpha_i\}_{0\le i\le 4}$
of expectations $\{\bar \alpha_i\}_{0\le i\le 4}$: if $\alpha_0=\alpha_1-\alpha_2\alpha_3\alpha_4$ then
\begin{eqnarray*}
\alpha_0-\bar \alpha_0&=&(\alpha_1-\bar \alpha_1)-(\alpha_2-\bar \alpha_2)\bar \alpha_3 \bar \alpha_4-(\alpha_3-\bar \alpha_3)\bar \alpha_2 \bar \alpha_4-(\alpha_4-\bar \alpha_4)\bar \alpha_2 \bar \alpha_3
\\
&&-(\alpha_2-\bar \alpha_2)(\alpha_3-\bar \alpha_3)  \alpha_4-(\alpha_2-\bar \alpha_2)(\alpha_4-\bar \alpha_4)  \bar \alpha_3-(\alpha_3-\bar \alpha_3)(\alpha_4-\bar \alpha_4)  \bar \alpha_2
\\
&&-\bar \alpha_2\bar \alpha_3\bar \alpha_4-\bar \alpha_0,
\end{eqnarray*}
so that by H\"older's inequality in probability (and using the boundedness assumption) 
\begin{equation}\label{0:1D-prod-sum}
|\var{\alpha_0}-\var{\alpha_1-\alpha_2\bar \alpha_3 \bar \alpha_4-\alpha_3\bar \alpha_2 \bar \alpha_4-\alpha_4\bar \alpha_2 \bar \alpha_3}|
\,\lesssim\, \sum_{i=1}^4 \expec{|\alpha_i-\bar \alpha_i|^3}.
\end{equation}
By the definition \eqref{0:1D-19} of $I_\e(f,g)$, we have
$
\sigma_\e^2(f,g) = \e^{-1} \var{\alpha_1-\alpha_2\alpha_3\alpha_4}
$
with $\alpha_1=\fint_0^{\frac 1\e} \frac{1}{a(x)}f(\e x)g(\e x)dx$, $\alpha_2=(\fint_0^{\frac1\e} \frac1a)^{-1}$, $\alpha_3= \fint_0^{\frac 1\e}\frac{1}{a(y)}g(\e y)dy$,
and  $\alpha_4=\fint_0^{\frac 1\e} \frac1{a(y)}f(\e y)dy$.
By~\eqref{0:1D-prod-sum}, this implies
$$
|\sigma_\e^2(f,g) - \e^{-1} \var{\alpha_1-\alpha_2\bar \alpha_3 \bar \alpha_4-\alpha_3\bar \alpha_2 \bar \alpha_4-\alpha_4\bar \alpha_2 \bar \alpha_3}|\,\lesssim\, \e^{-1} \sum_{i=1}^4 \expec{|\alpha_i-\bar \alpha_i|^3}.
$$
As a by-product of the proof of Theorem~\ref{0:1D-thm-quant1} we have $\expec{|\alpha_i-\bar \alpha_i|^3}\lesssim \sqrt{\e}^3$ for all $1\le i\le 4$ so that
the above turns into 
$$
|\sigma_\e^2(f,g) - \e^{-1} \var{\alpha_1-\alpha_2\bar \alpha_3 \bar \alpha_4-\alpha_3\bar \alpha_2 \bar \alpha_4-\alpha_4\bar \alpha_2 \bar \alpha_3}|\,\lesssim\, \sqrt{\e}.
$$
It remains to argue that one can replace $(\alpha_2,\bar \alpha_2)$ by $(\beta_2,\bar \beta_2)=(2\bar a - {\bar a}^2\fint_0^{\frac1\e} \frac1a,\bar a)$.
Indeed, for any three random variables $Y_1,Y_2,Y_3=Y_1-Y_2$, we have
$
\var{Y_1}=\var{Y_2+Y_3}=\var{Y_2}+\var{Y_3}+2\cov{Y_2}{Y_3},
$
so that by Cauchy-Schwarz' inequality
$$
|\var{Y_1}-\var{Y_2}|\le \var{Y_3}+2\var{Y_2}^\frac12\var{Y_3}^\frac12.
$$
We then apply this inequality with  $Y_1=\alpha_1-\alpha_2\bar \alpha_3 \bar \alpha_4-\alpha_3\bar \alpha_2 \bar \alpha_4-\alpha_4\bar \alpha_2 \bar \alpha_3$, $Y_2=\alpha_1-\beta_2\bar \alpha_3 \bar \alpha_4-\alpha_3\bar \beta_2 \bar \alpha_4-\alpha_4\bar \beta_2 \bar \alpha_3$ and $Y_3=Y_1-Y_2=-(\alpha_2-\beta_2)\bar \alpha_3 \bar \alpha_4-\alpha_3(\bar \alpha_2-\bar \beta_2) \bar \alpha_4-\alpha_4(\bar \alpha_2-\bar \beta_2) \bar \alpha_3$.
On the one hand, by the boundedness of the $\alpha$'s, and  \eqref{0:1D-rewriting-var0} \&~\eqref{0:1D-rewriting-var1} in Step~1,
$$
\var{Y_3}\le \expec{Y_3^2} \lesssim \expec{\mu_\e^2} \,\lesssim\, \e^2.
$$
On the other hand, by the proof of Theorem~\ref{0:1D-thm-quant1} as above and the boundedness of the $\bar \alpha$'s, 
$$
\var{Y_2} \lesssim \sum_i \var{\alpha_i} \,\lesssim\, \e.
$$
Hence, the combination of the last four estimates yields
$$
|\sigma_\e^2(f,g) - \e^{-1} \var{\alpha_1-\beta_2\bar \alpha_3 \bar \alpha_4-\alpha_3\bar \beta_2 \bar \alpha_4-\alpha_4\bar \beta_2 \bar \alpha_3}|\,\lesssim\, \sqrt{\e},
$$
which can be reformulated in form of~\eqref{0:1D-thm5-1.2}.

\medskip

Since all the $X_{i,\e}$ are spatial averages of $\frac1a$ with deterministic test-functions, it is enough
to study the covariance of such averages to characterize the asymptotic behavior of $\sigma_\e(f,g)$, which
we do in the following step.

\medskip

\step{3} Proof of~\eqref{e.conv-cov-int} and \eqref{0:1D-lim-var}.

We claim that it is enough to establish that for all $h \in C^1([0,1])$, we have
\begin{equation}\label{0:1D-thm5-1.3+}
\Big|\e^{-1}\var{\fint_0^\frac1\e \frac1{a(x)} h (\e x)dx}- \calQ \int_0^1 h^2\Big|\,\lesssim\, \|h\|_{C^1([0,1])}^2\e^{1 \wedge (\beta-1)}.
\end{equation}
By polarization, this indeed yields for all $h_1,h_2 \in  C^1([0,1])$
\begin{multline}\label{0:1D-thm5-1.3}
\Big|\e^{-1}\cov{\fint_0^\frac1\e \frac1{a(x)} h_1(\e x)dx}{\fint_0^\frac1\e \frac1{a(x)} h_2(\e x)dx}- \calQ \int_0^1 h_1h_2\Big|\\
\lesssim\,( \|h_1\|_{C^1([0,1])}^2+\|h_2\|_{C^1([0,1])}^2 )\e^{1 \wedge (\beta-1)},
\end{multline}
so that, by \eqref{0:1D-thm5-1.2} and applying \eqref{0:1D-thm5-1.3} to each of the 16 terms $\cov{X_{i,\e}}{X_{j,\e}}$ for $1\le i,j\le 4$,
we obtain the estimate
\begin{equation}\label{0:1D-thm5-1.4}
|\sigma_\e^2(f,g)- \sigma^2(f,g)|\,\lesssim\, ( \|h_1\|_{C^1([0,1])} +\|h_2\|_{C^1([0,1])}  )\e^{\frac12 \wedge(\beta-1)}
\end{equation}
where 
$$
\sigma^2(f,g)=\calQ(\overline{f^2g^2}-3\bar f^2\bar g^2+\bar f^2 \overline{g^2}+\overline{f^2}\bar g^2+4\overline{fg}\bar f\bar g-2\overline {fg^2}\bar f-2\overline{f^2g}\bar g),
$$
and $\overline h$ is a short-hand notation for $\int_0^1h$. 
It remains to notice that this definition of $\sigma^2(f,g)$ coincides with \eqref{0:1D-lim-var} (by expanding the product).

\medskip

We now turn to the proof of \eqref{0:1D-thm5-1.3+}. 
Set $c:x\mapsto \cov{\frac1a(x)}{\frac1a(0)}$. 
By the Helffer-Sj\"ostrand representation formula in the weaker form of the upper bound
$$
|\cov{X}{Y}|\,\lesssim\, \iint_{\R\times \R} \expec{|D_{z_1} X|^2}^\frac12\expec{|D_{z_2}Y|^2}^\frac12|\calC(z_1-z_2)|dz_1dz_2,
$$
we obtain for $X=\frac1a(x)$ and $Y=\frac1a(0)$ using Lemma~\ref{eq:mom-a}
\begin{eqnarray*}
|c(x)|&\lesssim&\iint_{\R\times \R} \expec{\exp(2G)}  \delta(z_1-x)\delta(z_2) |\calC(z_1-z_2)|dz_1dz_2 
\\
&\lesssim&  |\calC(x)|,
\end{eqnarray*}
so that $c$ has the same decay properties as $\calC$. The use of the Helffer-Sj\"ostrand representation formula (or here, covariance estimate) is definitely an overkill in dimension 1, and \eqref{0:1D-thm5-1.3+} can be obtained directly as a consequence of the explicit formula \eqref{e:fo-cov} below.

\medskip

Extend $h$ by zero outside $[0,1]$. 
By splitting $\int h^2 = \frac12 \int_\R( h(x_1)^2+h(x_1+\e x_2)^2) dx_1$ for any $x_2 \in \R$
and reconstructing a square,
a direct calculation yields
\begin{eqnarray}
\lefteqn{\Big|\e^{-1}\var{\fint_0^\frac1\e \frac1{a(x)} h(\e x)dx}-\calQ\int h^2\Big|}\nonumber 
\\
&=&  \Big|\iint_{\R\times \R}   h(x_2+\e x_1)h(x_2) c(x_1)dx_1dx_2-\frac12 \iint_{\R\times \R} (h(x_2)^2+h(x_2+\e x_1)^2)c(x_1)dx_2dx_1\Big|\nonumber
\\
&\le &\frac12 \Big|\iint_{\R\times \R} (h(x_2)-h(x_2+\e x_1))^2 c(x_1)dx_2dx_1\Big|.\label{e.cov-conv0}
\end{eqnarray}
Since $h$ is smooth in $C^1([0,1])$ and vanishes outside $[0,1]$, we have for all $x_1,x_2$%
\begin{multline*}
|h(x_2)-h(x_2+\e x_1)| \le \e |x_1| \| h'\|_{L^\infty}^2 \mathds 1_{x_2 \in [0,1]} \mathds 1_{x_2 + \e x_1 \in [0,1]} 
\\
+ \|h\|_{L^\infty}^2 \Big(\mathds 1_{x_2 \notin [0,1]} \mathds 1_{x_2 +\e x_1 \in [0,1]}+\mathds 1_{x_2 \in  [0,1]} \mathds 1_{x_2 +\e x_1 \notin [0,1]} \Big),
\end{multline*}
so that
\begin{multline*}
\iint_{\R\times \R} (h(x_2)-h(x_2+\e x_1))^2 |c(x_1)|dx_2dx_1
\,\lesssim_h \, \int_{x_2 \in [0,1]}\int_{x_2+\e x_1 \in [0,1]} \e  |c(x_1)| dx_1dx_2
\\
+ \int_{x_2 \in [0,1]} \int_{x_2+\e x_1 \notin [0,1]} |c(x_1)| dx_1 dx_2
+ \int_{x_2 \notin [0,1]} \int_{x_2+\e x_1 \in [0,1]} |c(x_1)| dx_1 dx_2 . 
\end{multline*}
Using that $|c(x)|\lesssim |\calC(x)|\lesssim (1+|x|)^{-\beta}$, this yields
$$
\iint_{\R\times \R} (h(x_2)-h(x_2+\e x_1))^2 |c(x_1)|dx_2dx_1\,\lesssim\,\|h\|_{C^1([0,1])}^2 (\e+\e^{\beta-1}),
$$
and therefore~\eqref{0:1D-thm5-1.3}.

\medskip

\step4 Non-degeneracy of $\calQ$.

Recall that $\calQ=\int_\R \cov{\tfrac1a(x)}{\tfrac1a(0)}dx$ with $a(x)=\exp(G(x))$.
Whereas $\int_\R \calC \ge 0$ as the integral of a covariance, we have not assumed that $\int_\R \calC \ne 0$. As a consequence of our choice of log-normal coefficients, we shall prove that the stronger property $\calQ>0$ holds for the integral of the covariance of $\frac1a$.
We start with an explicit formula for $\cov{\tfrac1a(x)}{\tfrac1a(0)}$.
By stationarity of $G$, $\expec{\exp(-G(y))}=\expec{\exp(-G(0))}$ for all $y\in \R$, so that
\begin{eqnarray*}
\lefteqn{\cov{\tfrac1a(x)}{\tfrac1a(0)} }
\\
&=& \expec{\Big(\exp(- G(x))-\expec{\exp(-G(x))}\Big)
\Big(\exp(-G(0))-\expec{\exp(-G(0) )}\Big)}
\\
&=&\expec{\exp(-(G(0)+G(x)))}- \expec{\exp(-G(0))}^2.
\end{eqnarray*}
For $x$ fixed, the random variable $G(0)+G(x)$ is  Gaussian.
Its mean is zero and, by stationarity, its variance is given by 
$$
\expec{(G(0)+G(x))^2}\,=\, \var{G(0)} + \var{G(x)} +2\cov{G(x)}{G(0)} = 2 (\calC(0)+\calC(x)).
$$
The calculation~\eqref{e.calc-exp} for centered Gaussian random variables
then yields
\begin{equation*} 
\expec{\exp(-(G(0)+G(x)))} \,=\,  \exp(\calC(0)+\calC(x)),
\quad \expec{\exp(-G(0))}^2\,=\, \exp(\calC(0)),
\end{equation*}
which entails 
\begin{equation}\label{e:fo-cov}
\cov{\tfrac1a(x)}{\tfrac1a(0)}  =  \exp(\calC(0))(\exp(\calC(x))-1).
\end{equation}
By our assumption, $|\calC(x)|\lesssim (1+|x|)^{-\beta}$, so that there exists 
a constant $C>0$ such that for all $x\in \R$, $-C \le \calC(x)\le C$.
Since $\zeta:t \mapsto e^{t}-1$ is convex  and its second-derivative satisfies $\zeta''(t) \ge \exp(-C)=: \gamma>0$ on $[-C,C]$, Jensen's inequality for the strongly convex function $\zeta$ yields for all $T>0$
$$
\fint_{-T}^T (\exp(\calC(x))-1)dx \, \ge \, \exp\Big(\fint_{-T}^T \calC(x) dx\Big)-1+\frac\gamma2 \fint_{-T}^T \Big( \calC(x)-\fint_{-T}^T \calC\Big)^2dx.
$$
Since $\calC$ is integrable, we have 
$$
2T\exp\Big(\fint_{-T}^T \calC(x) dx\Big) =2T +\int_{-T}^T \calC +O(\frac1T),
$$
so that multiplying the above by $2T$ and taking the limit $T\uparrow +\infty$ yield
$$
\int_\R (\exp(\calC(x))-1)dx \,\ge\, \int_\R \calC+\frac\gamma2\int_\R \calC^2.
$$
By stationarity, the decay of $\calC$ and the Lebesgue dominated convergence theorem,
we have 
\begin{multline*}
0 \le 2L \var{\fint_{-L}^L G(x)dx} = 2L\fint_{-L}^L \fint_{-L}^L \cov{G(x)}{G(x')}dxdx'
\\
= 2L \fint_{-L}^L \fint_{-L}^L \calC(x-x') dxdx' \stackrel{L\uparrow +\infty}\longrightarrow \int_\R \calC ,
\end{multline*}
which implies that $\int_\R \calC+\frac\gamma2\int_\R \calC^2>0$, and
thus the claimed positivity of 
$$\calQ=\exp(\calC(0)) \int_{-\infty}^\infty (\exp(\calC(x))-1)dx \ge \exp(\calC(0))\Big(\int_\R \calC+\frac\gamma2\int_\R \calC^2\Big)>0.$$

\subsubsection{Asymptotic normality: proof of Proposition~\ref{prop:NA-int}}

\noindent To prove that fluctuations are asymptotically Gaussian, we use the Stein-Malliavin method.
We split the proof into two steps.  

\medskip

\step{1} Formulas for the first and second Malliavin derivatives of $I_\e(f,g)$.
\\
The Malliavin derivative is given by \eqref{0:1D-18} in the form
\begin{equation*}  
D_z I_\e(f,g)=\e \frac{1}{a(z)}K_{0,\e}^z(f,g),
\end{equation*}
where $K_{0,\e}^z(f,g)$ reads
\begin{eqnarray*}  
K_{0,\e}^z(f,g)&=&- f(\e z)g(\e z) \mathds{1}_{z\in (0,\frac1\e)}  +f(\e z)  \mathds{1}_{z\in (0,\frac1\e)} \bigg(\fint_0^{\frac1\e} \frac1{a( y)}dy\bigg)^{-1}  \fint_0^{\frac 1\e}\frac{1}{a(y)}g(\e y)dy 
\\
&&+g(\e z)  \mathds{1}_{z\in (0,\frac1\e)} \bigg(\fint_0^{\frac1\e} \frac1{a( y)}dy\bigg)^{-1} \fint_0^{\frac 1\e} \frac1{a(y)}f(\e y)dy
\\
&&- \mathds{1}_{z\in (0,\frac1\e)}\bigg(\fint_0^{\frac1\e} \frac1{a( y)}dy\bigg)^{-2}  \fint_0^{\frac 1\e}\frac{1}{a(y)}g(\e y)dy\fint_0^{\frac 1\e} \frac1{a(y)}f(\e y)dy.
\end{eqnarray*}
Differentiating again and using the Leibniz rule, this yields for the second Malliavin derivative
\begin{equation}  \label{0:1D-Mall-dec-fin}
D^2_{z,z'} I_\e(f,g)=-\e\delta(z-z') \frac{1}{a(z)}K_{0,\e}^z(f,g)
+\e^2 \frac{1}{a(z)a(z')} K_{1,\e}^{z,z'}(f,g),
\end{equation}
where $K_{1,\e}^{z,z'}(f,g)$ is given by  
\begin{eqnarray*}
K_{1,\e}^{z,z'}(f,g)&=&
-2f(\e z)  \mathds{1}_{z\in (0,\frac1\e)} \mathds{1}_{z'\in (0,\frac1\e)} \bigg(\fint_0^{\frac1\e} \frac1{a( y)}dy\bigg)^{-2}  \fint_0^{\frac 1\e}\frac{1}{a(y)}g(\e y)dy 
\\
&&-2 g(\e z)\mathds{1}_{z\in (0,\frac1\e)} \mathds{1}_{z'\in (0,\frac1\e)} \bigg(\fint_0^{\frac1\e} \frac1{a( y)}dy\bigg)^{-2}  \fint_0^{\frac 1\e}\frac{1}{a(y)}f(\e y)dy 
\\
&&+(f(\e z)g(\e z')+f(\e z')g(\e z)) \mathds{1}_{z\in (0,\frac1\e)} \mathds{1}_{z'\in (0,1)} \bigg(\fint_0^{\frac1\e} \frac1{a( y)}dy\bigg)^{-1}
\\
&&+2 \mathds{1}_{z\in (0,\frac1\e)} \mathds{1}_{z'\in (0,\frac1\e)}\bigg(\fint_0^{\frac1\e} \frac1{a( y)}dy\bigg)^{-3}  \fint_0^{\frac 1\e}\frac{1}{a(y)}g(\e y)dy\fint_0^{\frac 1\e} \frac1{a(y)}f(\e y)dy.
\end{eqnarray*}

\medskip

\step{2} Application of the second-order Poincar\'e inequality~\eqref{e.SndPI}.
\\
By~\eqref{e.SndPI} we have (as soon as $\sigma_\e(f,g)\ne 0$)
\begin{equation}\label{e.decompX1X2}
d_{TV}\Big(\frac{I_\e(f,g)-\expec{I_\e(f,g)}}{\sqrt \e \sigma_\e(f,g)},\Nc\Big)
\,\lesssim\, \tfrac1\e \sigma_\e(f,g)^{-2}\,\expec{\3D^2I_\e(f,g)\3_{\beta}^4}^\frac14\expec{\|DI_\e(f,g)\|_\beta^4}^\frac14.
\end{equation}
As in the proof of Theorem~\ref{0:1D-thm-quant2} we have
\begin{equation}\label{e.controlX0}
\expec{\|DI_\e(f,g)\|_\beta^4}^\frac14 \,\lesssim \, \sqrt \e \|f\|_{L^4(0,1)}\|g\|_{L^4(0,1)},
\end{equation}
and it remains to estimate the operator norm of the second Malliavin derivative.
For $\beta>1$, recall that 
$$
\3 X \3_\beta = \sup  \Big\{ \Big|\iint_{\R \times \R}    \zeta(x)  \zeta(y) X(x,y)dxdy\Big|\,:\,  \|\zeta\|_{L^{2,\infty}(\R)}=1  \Big\}  .
$$
We shall treat this norm differently for the two RHS terms of~\eqref{0:1D-Mall-dec-fin}.
For the first term $X_1:=-\e\delta(z-z') \frac{1}{a(z)}K_{0,\e}^z(f,g)$, we have
$$
\3 X_1 \3_\beta \,=\, \e \sup  \Big\{ \Big|\int_{\R}    \zeta^2(z) \frac{1}{a(z)}K_{0,\e}^z(f,g) dz\Big|\,:\,  \|\zeta\|_{L^{2,\infty}(\R)}=1  \Big\}.
$$
We then take local suprema of $\zeta$ and use H\"older' inequality with exponent $q> 1$ to the effect of
\begin{equation*}
\Big|\int_{\R}    \zeta^2(z) \frac{1}{a(z)}K_{0,\e}^z(f,g) dz\Big|
\, \le\, \Big(\int_\R \big( \sup_{B(z)} |\zeta|\big)^\frac{2q}{q-1}dz\Big)^{ 1-\frac 1q} \Big(\int_{\R} \frac{1}{a(z)^q}\big(K_{0,\e}^z(f,g)\big)^q dz\Big)^\frac1q.
\end{equation*}
Using that the spaces $L^{p,\infty}(\R)$ are nested (for the same reason as for the discrete $\ell^p(\Z)$ spaces) in form of $\|\zeta\|_{L^{\frac{2q}{q-1},\infty}(\R)}\lesssim \|\zeta\|_{L^{2,\infty}(\R)} = 1$, this yields
$$
\3 X_1 \3_\beta \,\lesssim\, \e \|\tfrac{1}{a}K_{0,\e}^z(f,g)\|_{L^q(\R)},
$$
so that, by a suitable use of Lemma~\ref{lem:Minko}, we obtain for some finite constant $C>0$ (independent of $q$)
$$
\expec{\3 X_1 \3_\beta^4}^\frac14 \,\lesssim\, \e \expec{ \|\tfrac{1}{a}K_{0,\e}^z(f,g)\|_{L^q(\R)}^4}^\frac14 \,\lesssim\, \e^{1-\frac1q} \exp(Cq) \|f\|_{L^4(0,1)}\|g\|_{L^4(0,1)}.
$$
Let us now optimize in $q> 1$ by choosing $q\sim |\log \e|^\frac12$, so that
\begin{equation}\label{e.controlX1}
\expec{\3 X_1\3_\beta^4}^\frac14 \,\lesssim\, \e \exp(C|\log\e|^\frac12) \|f\|_{L^4(0,1)}\|g\|_{L^4(0,1)}.
\end{equation}
We now turn to the second RHS term $X_2:=\e^2 \frac{1}{a(z)a(z')} K_{1,\e}^{z,z'}(f,g)$
 of~\eqref{0:1D-Mall-dec-fin}, for which we have
$$
\3 X_2 \3_\beta = \e^2  \sup  \Big\{ \Big|\iint_{\R \times \R}    \zeta(z)  \zeta(z')  \frac{1}{a(z)a(z')} K_{1,\e}^{z,z'}(f,g)dzdz'\Big|\,:\,  \|\zeta\|_{L^{2,\infty}(\R)}=1  \Big\} .
$$
By Cauchy-Schwarz' inequality, this directly implies
$$
\3 X_2 \3_\beta \le \e^2 \Big( \iint_{\R \times \R}   \frac{1}{a(z)^2a(z')^2} (K_{1,\e}^{z,z'}(f,g))^2dzdz'\Big)^\frac12,
$$
and therefore, using Lemma~\ref{lem:Minko} again,
\begin{equation}\label{e.controlX2}
\expec{\3 X_2 \3_\beta^4}^\frac14 \,\lesssim \, \e  \|f\|_{L^4(0,1)} \|g\|_{L^4(0,1)} .
\end{equation}
The combination of \eqref{e.controlX0}, \eqref{e.controlX1}, and \eqref{e.controlX2} then allows to turn \eqref{e.decompX1X2}
into the claim~\eqref{e:NA-int}.

\subsubsection{Quantitative CLT: proof of Theorem~\ref{0:1D-thm-quant3}}

In view of Proposition~\ref{prop:NA-int}, it remains to replace $\sigma_\e(f,g)$
by $\sigma(f,g)$ in \eqref{e:NA-int} using Proposition~\ref{prop:cov-int}.

\medskip

Since $\sigma(f,g)>0$ and $\sigma_\e(f,g)\ge 0$,  by \eqref{0:1D-thm5-1.4} 
we have 
\begin{equation}\label{0:1D-thm5-3}
|\frac{\sigma_\e(f,g)}{\sigma(f,g)}-1|=\frac{|\sigma_\e^2(f,g)-\sigma^2(f,g)|}{\sigma(f,g)(\sigma_\e(f,g)+\sigma(f,g))} \,\le \, \frac{|\sigma_\e^2(f,g)-\sigma^2(f,g)|}{ \sigma^2(f,g)} \,\lesssim\,\e^{\frac12 \wedge (\beta-1)}.
\end{equation}
Next, we use the definition of the total variation distance. 
Let $\nu_\e$ denote the law of $\frac{I_\e(f,g)- \expec{I_\e(f,g)}}{\sqrt\e\sigma_\e(f,g)}$ and $\tilde \nu_\e$
denote the law of $\frac{I_\e(f,g)- \expec{I_\e(f,g)}}{\sqrt\e\sigma(f,g)}$.
Then for all Borel sets $A$ of $\R$ we have $\tilde \nu_\e(A)= \nu_\e(\frac{\sigma_\e(f,g)}{\sigma(f,g)} A)$.
Let $\mathcal N_\tau$ be a centered Gaussian variable of variance $\tau^2>0$ on $\R$.
Its density is given by $t\mapsto \frac1{\tau\sqrt{2\pi}} e^{-\frac{t^2}{\tau^2}}$, so that for all $A$,
the change of variables  $t \leadsto  \frac s\tau$ implies the identity 
\begin{equation}\label{e.resc-N}
\mathcal N(\tau A):=\int_{\tau A} \frac1{\sqrt{2\pi}} e^{-t^2}dt = \int_{A} \frac1{\tau \sqrt{2\pi}} e^{-\frac{s^2}{\tau^2}}ds=\mathcal N_\tau (A)
\end{equation}
(where $\mathcal N$ still denotes the centered Gaussian variable of variance unity).
By  the triangle inequality we have for all $A$
\begin{eqnarray*}
|\tilde \nu_\e(A)-\mathcal N(A)|&=& | \nu_\e(\frac{\sigma_\e(f,g)}{\sigma(f,g)} A)-\mathcal N(A)|
\\
&\le&  | \nu_\e(\frac{\sigma_\e(f,g)}{\sigma(f,g)} A)-\mathcal N(\frac{\sigma_\e(f,g)}{\sigma(f,g)} A)|
+|\mathcal N(\frac{\sigma_\e(f,g)}{\sigma(f,g)} A)-\mathcal N(A)|
\\
&=& | \nu_\e(\frac{\sigma_\e(f,g)}{\sigma(f,g)} A)-\mathcal N(\frac{\sigma_\e(f,g)}{\sigma(f,g)} A)|
+|\mathcal N_{\tau_\e}(A)-\mathcal N(A)|
\end{eqnarray*}
with the notation $\tau_\e:=\frac{\sigma_\e(f,g)}{\sigma(f,g)}$.
Bounding the RHS by taking the supremum over $A$ in the RHS and using that the total variation distance between two probability laws $\nu$ and $\tilde \nu$ is precisely given by
$
d_{TV}(\nu,\tilde \nu)\,:=\,\sup_{A'} |\nu(A')-\tilde \nu('A)|,
$
where the supremum runs over all Borel sets $A'$ of $\R$,
this yields
$$
|\tilde \nu_\e(A)-\mathcal N(A)| \le d_{TV}(\nu_\e,\mathcal N)+d_{TV}(\mathcal N_{\tau_\e},\mathcal N).
$$
Hence, with \eqref{e:NA-int} to bound the first RHS term, and estimating the distance between Gaussian laws via
$$
d_{TV}(\mathcal N_{\tau_\e},\mathcal N)\,=\,\frac12 \int_\R |\frac1{\tau_\e\sqrt{2\pi}} e^{-\frac{t^2}{\tau_\e^2}}-\frac1{\sqrt{2\pi}} e^{-{t^2}}|dt \,\lesssim\, 
|1-\tau_\e| \,\stackrel{\eqref{0:1D-thm5-3}}{\lesssim}\, \e^{\frac12 \wedge (\beta-1)},
$$
 the theorem follows.

\subsection{Scaling law for non-integrable correlations}

In this paragraph we place ourselves in the case of non-integrable covariance function, that is, for $\beta\le1$.
Whereas the normal approximation result and its proof are quite robust, we do not necessarily have existence of a unique limiting covariance structure. To have such a property we further assume the asymptotic homogeneity of the covariance function $\calC$.
\begin{hypo}[Asymptotic homogeneity of $\calC$]\label{hypo:as-ho}
Let $\gamma:[0,+\infty) \to [0,+\infty)$ be bounded and such that $\lim_{t \to 0} \gamma(t)=0$.
For $\beta=1$, we assume there exists $\bar \calC \in \R$ such that for all $L>0$
$$
\Big| \frac1{\log L} \int_{-L}^L \calC(x)dx - \bar \calC \Big| \le \gamma(\tfrac1L).
$$
For $0<\beta<1$, we assume there exist $\bar \calC^+,\bar\calC^- \in \R$ such that for all $x\ge 0$   
$$
\big| x^\beta \calC(\pm x) - \bar \calC^{\pm}\big| \le \gamma(\tfrac1x).
$$
\end{hypo} 
\begin{prop}\label{prop:cov-non-int}
Let $\beta\le 1$,  $G \not\equiv 0$ (otherwise $a \equiv 1$ is constant), and assume Hypothesis~\ref{hypo:as-ho}.
Let $f,g \in C^1([0,1])$, let $u_\e$ be the solution of  \eqref{0:1D-1},  $v_\e$
be the solution of  \eqref{0:1D-1} for $f$ replaced by $g$, and set $I_\e(f,g):=\int_0^1 u_\e'(x)g(x)dx=-\int_0^1 v_\e'(x)f(x)dx$ and  $\sigma_\e(f,g):=\pi(\e)^{-1} \sqrt{\var{I_\e(f,g)}}$.
Then we have 
\begin{equation}\label{e.conv-cov-nonint}
|\sigma_\e(f,g)^2-\sigma(f,g)^2|\,\lesssim\, \Pi_\beta(\e),
\end{equation}
where 
\begin{equation}
\Pi_\beta(\e):=
\begin{cases}
\e^{\beta\land (1-\beta)} + \inf_{\delta\in(0,1]}(\delta^{1-\beta} +\gamma (\delta^{-1}\e)) ,&\text{if $0<\beta<1$ and }\beta\neq \frac12,\\
\e^{\frac12}|\log\e| +  \inf_{\delta\in(0,1]}(\delta^{1-\beta} +\gamma (\delta^{-1}\e)) ,&\text{if $\beta=\frac12$},\\
|\log\e|^{-1}+\gamma(\e) ,&\text{if $\beta=1$},\\
\end{cases}
\end{equation}
and $\sigma^2(f,g):=  \calQ_\beta \Big((f-\bar f)(g-\bar g)\Big)$ with $\calQ_\beta$ a quadratic form defined as 
\begin{equation}
\calQ_\beta(h):=\exp(\calC(0))
\begin{cases}
\int_0^1\int_0^1 h(x)\frac{\bar \calC^{\sign(x-y)}}{|x-y|^{-\beta}}h(y)dxdy ,&\text{if $0<\beta<1$ },\\
\bar \calC \int_0^1 h^2(x)dx,&\text{if  $\beta=1$ }.\\
\end{cases}
\end{equation}
\end{prop}
We now turn to the normal approximation.
\begin{prop}\label{prop:NA-non-int}
Let $\beta\le 1$.
Let $f,g \in C^1([0,1])$, let $u_\e$ be the solution of  \eqref{0:1D-1},  $v_\e$
be the solution of  \eqref{0:1D-1} for $f$ replaced by $g$, and set $I_\e(f,g):=\int_0^1 u_\e'(x)g(x)dx=-\int_0^1 v_\e'(x)f(x)dx$ and $\sigma_\e(f,g):=\pi_\beta(\e)^{-1}\sqrt{\var{I_\e(f,g)}}$.
Then, as long as $\sigma_\e(f,g)\ne 0$, we have
\begin{equation}\label{e:NA-non-int}
d_{TV}\Big(\frac{I_\e(f,g)-\expec{I_\e(f,g)}}{\pi_\beta(\e) \sigma_\e(f,g)},\mathcal N \Big) \,\lesssim_{f,g}\, \frac1{\sigma_\e(f,g)^2}\pi_\beta(\e),
\end{equation}
where $\mathcal N$ is a centered normal variable of variance unity,
and $d_{TV}$ denotes the total variation distance between probability measures.
\end{prop}
As in the case of integrable covariance, these two propositions combine into the following quantitative central limit theorem.
\begin{theor}\label{0:1D-thm-quant3-non-int}
Let $\beta>1$, and $G \not\equiv 0$.
Let $f,g \in C^1([0,1])$, let $u_\e$ be the solution of  \eqref{0:1D-1},  $v_\e$
be the solution of  \eqref{0:1D-1} for $f$ replaced by $g$,  set $I_\e(f,g):=\int_0^1 u_\e'(x)g(x)dx=-\int_0^1 v_\e'(x)f(x)dx$ and $\sigma_\e(f,g):=\e^{-1}\var{I_\e(f,g)}$, 
and let $\sigma(f,g)$ be given by \eqref{0:1D-lim-var}.
If both $f$ and $g$ are not constant (otherwise $u_\e\equiv 0$ or $v_\e\equiv 0$ so that $I_\e(f,g)\equiv 0$), then $\sigma(f,g) >0$ and 
$I_\e(f,g)$ satisfies
\begin{equation}
d_{TV}\Big(\frac{I_\e(f,g)-\expec{I_\e(f,g)}}{\sqrt\e \sigma(f,g)},\mathcal N \Big) \,\lesssim_{f,g}\, \Pi_\beta(\e)+\pi_\beta(\e).
\end{equation}
\end{theor}

\subsubsection{Proof of Proposition~\ref{prop:cov-non-int}}

The same arguments as in the first two steps of the proof of Proposition~\ref{prop:cov-int}
yield, with the same notation, 
\begin{equation}\label{0:1D-thm5-1.2+}
|\sigma_\e^2(f,g)-\var{X_{1,\e}+X_{2,\e}-X_{3,\e}-X_{4,\e}}| \,\lesssim \,
\begin{cases}
\e^{\beta},&\text{if $0<\beta<1$},\\
 \e |\log(\e)| ,&\text{if $\beta=1$},\\
\end{cases}
\end{equation}
using Theorem~\ref{0:1D-thm-quant2} with $\beta \le 1$ rather than $\beta>1$.
It remains to establish a replacement for the estimate \eqref{0:1D-thm5-1.3+} valid for non-integrable covariance, and to investigate the non-degeneracy.
In this proof we choose to rely extensively on the explicit form of $c$ given
by \eqref{e:fo-cov} in the log-normal setting. 
This makes proofs shorter, but it is not essential.

\medskip

\step1 Proxy for \eqref{0:1D-thm5-1.3+}.
\begin{eqnarray*}
\var{\fint_0^\frac1\e \frac1{a(x)} h(\e x)dx}&=& \fint_0^\frac1\e\fint_0^\frac1\e c(x-y)h(\e x)h(\e y)dxdy
\\
&\stackrel{\eqref{e:fo-cov}}=& \exp(\calC(0))  \int_0^1\int_0^1 (\exp(\calC(\tfrac{x-y}\e))-1)h( x)h( y)dxdy.
\end{eqnarray*}
Since $\calC$ is bounded, we have $|\exp(\calC(x))-1-\calC(x)|\lesssim\calC(x)^2$, so that 
the above yields
\begin{multline*}
\Big|\var{\fint_0^\frac1\e \frac1{a(x)} h(\e x)dx}-\exp(\calC(0))  \int_0^1\int_0^1 \calC(\tfrac{x-y}\e)h( x)h( y)dxdy\Big|\\
\lesssim\, \int_0^1\int_0^1 \calC(\tfrac{x-y}\e)^2 |h( x)||h( y)|dxdy.
\end{multline*}
It remains to prove the following two estimates to conclude: for any $\delta\in(0,1]$, 
\begin{eqnarray*}
{\Big|\int_0^1\int_0^1 \e^{-\beta} \calC(\tfrac{x-y}\e)h( x)h( y)dxdy-\int_0^1\int_0^1 \frac{\bar \calC^{\sign (x-y)}}{|x-y|^{\beta}} h( x)h( y)dxdy\Big|}
&\lesssim& \delta^{1-\beta} +\gamma(\delta^{-1}\e),
\\
\int_0^1\int_0^1 \e^{-\beta} \calC(\tfrac{x-y}\e)^2|h( x)||h( y)|dxdy
&\lesssim & \e^{\beta\land(1-\beta)},
\end{eqnarray*}
where the second inequality comes with a correcting term $|\log \e|$ while $\beta=\frac12$. When $\beta<1$, the first inequality comes from direct calculation: extend $h$ by $0$ outside of $[0,1]$ and by change of variables $x_1=x$, $x_2=\frac{x-y}{\epsilon}$,
\begin{eqnarray*}
& &{\Big|\int_0^1\int_0^1 \e^{-\beta} \calC(\tfrac{x-y}\e)h( x)h( y)dxdy-\int_0^1\int_0^1 \frac{\bar \calC^{\sign (x-y)}}{|x-y|^{\beta}} h( x)h( y)dxdy\Big|} 
\\
&\le &\e^{1-\beta} {\int\int_{\R^2} \left|\calC(x_2)-\bar \calC^{\sign (x_2)}{|x_2|^{-\beta}}\right| |h(x_1)||h( x_1-\e x_2)|dx_1dx_2} 
\\
&\le &\e^{1-\beta} {\int\int_{\R^2}[\gamma(\delta^{-1}\e){|x_2|^{-\beta}}\mone_{|\e x_2|>\delta}+|x_2|^{-\beta}\mone_{|\e x_2|<\delta}]   |h(x_1)||h( x_1-\e x_2)|dx_1dx_2} 
\\
&\lesssim_h &\e^{1-\beta} {\int_0^1 dx_1\int_{\frac{x_1-1}{\e}}^{\frac{x_1}{\e}}[\gamma(\delta^{-1}\e){|x_2|^{-\beta}}\mone_{|\e x_2|>\delta}+|x_2|^{-\beta}\mone_{|\e x_2|<\delta}]  dx_2} 
\\
&\lesssim_\beta &\e^{1-\beta} [\gamma(\delta^{-1}\e) \e^{\beta-1}+(\e^{-1}\delta)^{1-\beta}] = \gamma(\delta^{-1}\e)+\delta^{1-\beta}.
\end{eqnarray*}

Hence if we choose $\e\ll \delta \ll 1$, then
\begin{multline}
{\Big|\int_0^1\int_0^1 \e^{-\beta} \calC(\tfrac{x-y}\e)h( x)h( y)dxdy-\int_0^1\int_0^1 \frac{\bar \calC^{\sign (x-y)}}{|x-y|^{\beta}} h( x)h( y)dxdy\Big|} \
 \lesssim\delta^{1-\beta} +\gamma(\delta^{-1}\e)
\end{multline}
will tend to $0$ as $\epsilon\to0$.

The second inequality comes from the decaying property of $\calC$. In fact, given that $|\calC(x)|\lesssim (1+|x|)^{-\beta}$, if $\beta\neq\frac12$,
\begin{eqnarray*}
\int_0^1\int_0^1 \e^{-2\beta} \calC(\tfrac{x-y}\e)^2|h( x)||h( y)|dxdy & \leq & \|h\|_\infty^2 \int_0^1\int_0^1\frac{dxdy}{(\epsilon+|x-y|)^{2\beta}}
\\
& \lesssim & \int _0^1 \int _{-1} ^1 \frac{dxdy}{(\e +|y|)^{2\beta}}\sim |(1+\epsilon)^{1-2\beta}-\epsilon^{1-2\beta}|.
\end{eqnarray*}
As a result, if $1-2\beta>0$,  the RHS stays bounded as $\epsilon$ tends to $0$.  However, if $1-2\beta<0$, the RHS tends to infinity at the speed of $\e^{1-2\beta}$. Combining these two facts,  for $\beta\neq \frac12$,
\begin{eqnarray*}
\int_0^1\int_0^1 \e^{-\beta} \calC(\tfrac{x-y}\e)^2|h( x)||h( y)|dxdy & \leq & \|h\|_\infty^2 \int_0^1\int_0^1\frac{\e^{\beta}dxdy}{(\epsilon+|x-y|)^{2\beta}}\, \lesssim\, \e^{\beta\land(1-\beta)}.
\end{eqnarray*}
And for the case $\beta=\frac12$, we have 
\begin{eqnarray*}
\int_0^1\int_0^1 \e^{-\beta} \calC(\tfrac{x-y}\e)^2|h( x)||h( y)|dxdy & \leq & \e^{\frac 12}\|h\|_\infty^2 \int_0^1\int_0^1\frac{dxdy}{\epsilon+|x-y|}\, \lesssim \, \epsilon ^{\frac12} |\log \e|.
\end{eqnarray*}

For the case where $\beta=1$, we prove 
\begin{eqnarray*}
\left|\int_0^1\int_0^1 \left(\e|\log\e|\right)^{-1} \calC_\e(\tfrac{x-y}\e)h( x)h( y)dxdy -\bar C_\e \int_0^1 h(x)^2dx\right|& \lesssim & |\log\e|^{-1}
\end{eqnarray*}
where $\bar C_\e$ is a truncated constant defined as 
$$\bar C_\e=\frac1{|\log\e|}\int_{\e|x|<1}\calC(x) dx$$
so that we can conclude by $|C_\e-\bar{\calC}|\le \gamma(\e)$, the asymptotic homogeneity assumption on $\calC$. In fact, we can adapt \eqref{e.cov-conv0} to see that 
\begin{eqnarray*}
& &\pi_\beta^{-1}(\e)\int_0^1\int_0^1 \calC(\tfrac{x-y}\e)h( x)h( y)dxdy -\bar C_\e \int_0^1 h(x)^2dx
\\
& = &|\log\e|^{-1}\Big[\int\int_{\mathbb{R}^2}\calC(x_2)h(x_1)h(x_1-\e x_2)dx_1dx_2
\\
&&\phantom{\frac12\int\int_{\R^2}}-\frac12\int\int_{\R^2}(h(x_1)^2+h(x_1-\e x_2)^2)\mone_{\e|x_2|<1}\calC(x_2) dx_1dx_2\Big]
\\
&=&|\log\e|^{-1}\int\int_{\R^2}(h(x_1)-h(x_1-\e x_2))^2 \calC(x_2)\mone_{\e|x_2|<1}dx_1dx_2,
\end{eqnarray*}
where in the last equality we use the fact that $x_1\in[0,1]$ and $x_1-\e x_2\in[0,1]$ imply $\e |x_2|<1$. Since $\e|x_2|<1$, we have
\begin{eqnarray*}
|h(x_1)-h(x_1-\e x_2)|&\lesssim_h& \mone_{x_1\in[0,1]}(\e\mone_{x_2\in[-\tfrac{x_1}\e,\tfrac{1-x_1}\e]}+\mone_{x_2\in [-\tfrac1\e,-\tfrac{x_1}\e]\cup[\tfrac{1-x_1}\e,\tfrac1\e]}) 
\\
& &+ \mone_{x_1\in[-1,0],x_2\in[-\tfrac{x_1}\e,\tfrac1\e]}+\mone_{x_1\in[1,2],x_2\in[-\tfrac1\e,\tfrac{1-x_1}\e]}
\end{eqnarray*}
Thus by $|\calC(x)|\lesssim (1+|x|)^{-1}$ ,
\begin{eqnarray*}
& &\left|\pi_\beta^{-1}(\e)\int_0^1\int_0^1 \calC(\tfrac{x-y}\e)h( x)h( y)dxdy -\bar C_\e \int_0^1 h(x)^2dx\right|
\\
& \lesssim_h & |\log\e|^{-1} (1+\e|\log\e|)\sim|\log\e|^{-1}.
\end{eqnarray*}
Finally we need to prove 
\begin{eqnarray*}
\int_0^1\int_0^1 \e^{-1} \calC(\tfrac{x-y}\e)^2|h( x)||h( y)|dxdy
&\lesssim & 1.
\end{eqnarray*}
This comes simply from the fact that $\calC(x)\lesssim(1+|x|)^{-1}$ is square-integrable:
\begin{eqnarray*}
&&\int\int_{\R^2} \e^{-1} \calC(\tfrac{x-y}\e)^2|h( x)||h( y)|dxdy
\\
&=&\int\int_{\R^2} \calC(x_2)^2|h( x_1)||h( x_1-\e x_2)|dxdy\\
&\lesssim_h& \int\int_{\R^2}  \calC(x_2)^2dxdy\lesssim 1.
\end{eqnarray*}
So we can conclude
\begin{equation}
|\sigma_\e^2(f,g)-\sigma^2(f,g)|\lesssim
\begin{cases}
\e^{\beta\land (1-\beta)} + \inf_\delta(\delta^{1-\beta} +\gamma (\delta^{-1}\e)) ,&\text{if $0<\beta<1$ and }\beta\neq \frac12,\\
\e^{\frac12}|\log\e| +  \inf_\delta(\delta^{1-\beta} +\gamma (\delta^{-1}\e)) ,&\text{if $\beta=\frac12$},\\
|\log\e|^{-1}+\gamma(\e) ,&\text{if $\beta=1$},\\
\end{cases}
\end{equation}
with 
\begin{equation}
\sigma^2(f,g)=\exp(\calC(0))
\int_0^1\int_0^1 (f(x)-\bar f )( g(x)-\bar g)\frac{\bar \calC^{\sign(x-y)}}{|x-y|^{-\beta}}(f(y)-\bar f )( g(y)-\bar g)dxdy
\end{equation}
for $\beta\in(0,1)$ and 
\begin{equation}
\sigma^2(f,g)=\bar C\exp(\calC(0))
\int_0^1 (f(x)-\bar f)^2( g(x)-\bar g)^2 dx
\end{equation}
while $\beta=1$.

\subsubsection{Proof of Proposition~\ref{prop:NA-non-int}}
We use the same argument as Proposition \ref{prop:NA-int}, with $\beta\le 1$.

First recall the Malliavin derivatives in Proposition \ref{prop:NA-int}
\begin{equation*}  
D_z I_\e(f,g)=\e \frac{1}{a(z)}K_{0,\e}^z(f,g),
\end{equation*}
where $K_{0,\e}^z(f,g)$ reads
\begin{eqnarray*}  
K_{0,\e}^z(f,g)&=&- f(\e z)g(\e z) \mathds{1}_{z\in (0,\frac1\e)}  +f(\e z)  \mathds{1}_{z\in (0,\frac1\e)} \bigg(\fint_0^{\frac1\e} \frac1{a( y)}dy\bigg)^{-1}  \fint_0^{\frac 1\e}\frac{1}{a(y)}g(\e y)dy 
\\
&&+g(\e z)  \mathds{1}_{z\in (0,\frac1\e)} \bigg(\fint_0^{\frac1\e} \frac1{a( y)}dy\bigg)^{-1} \fint_0^{\frac 1\e} \frac1{a(y)}f(\e y)dy
\\
&&- \mathds{1}_{z\in (0,\frac1\e)}\bigg(\fint_0^{\frac1\e} \frac1{a( y)}dy\bigg)^{-2}  \fint_0^{\frac 1\e}\frac{1}{a(y)}g(\e y)dy\fint_0^{\frac 1\e} \frac1{a(y)}f(\e y)dy.
\end{eqnarray*}
And
\begin{equation}  
D^2_{z,z'} I_\e(f,g)=-\e\delta(z-z') \frac{1}{a(z)}K_{0,\e}^z(f,g)
+\e^2 \frac{1}{a(z)a(z')} K_{1,\e}^{z,z'}(f,g),
\end{equation}
where $K_{1,\e}^{z,z'}(f,g)$ is given by  
\begin{eqnarray*}
K_{1,\e}^{z,z'}(f,g)&=&
-2f(\e z)  \mathds{1}_{z\in (0,\frac1\e)} \mathds{1}_{z'\in (0,\frac1\e)} \bigg(\fint_0^{\frac1\e} \frac1{a( y)}dy\bigg)^{-2}  \fint_0^{\frac 1\e}\frac{1}{a(y)}g(\e y)dy 
\\
&&-2 g(\e z)\mathds{1}_{z\in (0,\frac1\e)} \mathds{1}_{z'\in (0,\frac1\e)} \bigg(\fint_0^{\frac1\e} \frac1{a( y)}dy\bigg)^{-2}  \fint_0^{\frac 1\e}\frac{1}{a(y)}f(\e y)dy 
\\
&&+(f(\e z)g(\e z')+f(\e z')g(\e z)) \mathds{1}_{z\in (0,\frac1\e)} \mathds{1}_{z'\in (0,1)} \bigg(\fint_0^{\frac1\e} \frac1{a( y)}dy\bigg)^{-1}
\\
&&+2 \mathds{1}_{z\in (0,\frac1\e)} \mathds{1}_{z'\in (0,\frac1\e)}\bigg(\fint_0^{\frac1\e} \frac1{a( y)}dy\bigg)^{-3}  \fint_0^{\frac 1\e}\frac{1}{a(y)}g(\e y)dy\fint_0^{\frac 1\e} \frac1{a(y)}f(\e y)dy.
\end{eqnarray*}

\medskip

By~\eqref{e.SndPI} we have (as soon as $\sigma_\e(f,g)\ne 0$)
\begin{equation}\label{e.decompX1X2-non-int}
d_{TV}\Big(\frac{I_\e(f,g)-\expec{I_\e(f,g)}}{\pi_\beta(\e) \sigma_\e(f,g)},\Nc\Big)
\,\lesssim\, \pi_\beta(\e)^{-2} \sigma_\e(f,g)^{-2}\,\expec{\3D^2I_\e(f,g)\3_{\beta}^4}^\frac14\expec{\|DI_\e(f,g)\|_\beta^4}^\frac14.
\end{equation}
As in the proof of Theorem~\ref{0:1D-thm-quant2} we have (since $\frac 2 {2-\beta} \le2$ when $\beta\leq 1$)
\begin{equation*}
\expec{\|DI_\e(f,g)\|_\beta^4}^\frac14 \,\lesssim \,  \pi_{\beta}(\epsilon) \|f\|_{L^4(0,1)}\|g\|_{L^4(0,1)},
\end{equation*}
and it remains to estimate the operator norm of the second Malliavin derivative.
For $\beta<1$, recall that 
$$
\3 X \3_\beta = \sup  \Big\{ \Big|\iint_{\R \times \R}    \zeta(x)  \zeta(y) X(x,y)dxdy\Big|\,:\,  \|\zeta\|_{L^{\frac{2}{\beta},\infty}(\R)}=1  \Big\}  .
$$
We shall treat this norm differently for the two RHS terms of~\eqref{0:1D-Mall-dec-fin}.
For the first term $X_1:=-\e\delta(z-z') \frac{1}{a(z)}K_{0,\e}^z(f,g)$, we have
$$
\3 X_1 \3_\beta \,=\, \e \sup  \Big\{ \Big|\int_{\R}    \zeta^2(z) \frac{1}{a(z)}K_{0,\e}^z(f,g) dz\Big|\,:\,  \|\zeta\|_{L^{\frac 2\beta,\infty}(\R)}=1  \Big\}.
$$
We then take local suprema of $\zeta$ and use H\"older' inequality with exponent $q> 1$ to the effect of
\begin{equation*}
\Big|\int_{\R}    \zeta^2(z) \frac{1}{a(z)}K_{0,\e}^z(f,g) dz\Big|
\, \le\, \Big(\int_\R \big( \sup_{B(z)} |\zeta|\big)^\frac{2q}{q-1}dz\Big)^\frac12 \Big(\int_{\R} \frac{1}{a(z)^q}\big(K_{0,\e}^z(f,g)\big)^q dz\Big)^\frac1q.
\end{equation*}
Using that the spaces $L^{p,\infty}(\R)$ are nested (for the same reason as for the discrete $\ell^p(\Z)$ spaces) in form of $\|\zeta\|_{L^{\frac{2q}{q-1},\infty}(\R)}\lesssim \|\zeta\|_{L^{\frac 2\beta,\infty}(\R)} = 1$, this yields, for $\frac{2q}{q-1}\ge\frac 2\beta$,
$$
\3 X_1 \3_\beta \,\lesssim\, \e \|\tfrac{1}{a}K_{0,\e}^z(f,g)\|_{L^q(\R)},
$$
so that, by a suitable use of Lemma~\ref{lem:Minko}, we obtain for some finite constant $C>0$ (independent of $q$)
$$
\expec{\3 X_1 \3_\beta^4}^\frac14 \,\lesssim\, \e \expec{ \|\tfrac{1}{a}K_{0,\e}^z(f,g)\|_{L^q(\R)}^4}^\frac14 \,\lesssim\, \e^{1-\frac1q} \exp(Cq) \|f\|_{L^4(0,1)}\|g\|_{L^4(0,1)}.
$$
Let us now optimize in $q> 1$ by choosing $q=\frac{1}{1-\beta}$, so that
\begin{equation} 
\expec{\3 X_1\3_\beta^4}^\frac14 \,\lesssim\, \e^{\beta} \|f\|_{L^4(0,1)}\|g\|_{L^4(0,1)}.
\end{equation}
We now turn to the second RHS term $X_2:=\e^2 \frac{1}{a(z)a(z')} K_{1,\e}^{z,z'}(f,g)$
 of~\eqref{0:1D-Mall-dec-fin}, for which we have
$$
\3 X_2 \3_\beta = \e^2  \sup  \Big\{ \Big|\iint_{\R \times \R}    \zeta(z)  \zeta(z')  \frac{1}{a(z)a(z')} K_{1,\e}^{z,z'}(f,g)dzdz'\Big|\,:\,  \|\zeta\|_{L^{\frac 2\beta,\infty}(\R)}=1  \Big\} .
$$
By H\"older inequality, this directly implies
$$
\3 X_2 \3_\beta \le \e^2 \Big( \iint_{\R \times \R}   \frac{1}{a(z)^2a(z')^2} (K_{1,\e}^{z,z'}(f,g))^{\frac 2{2-\beta}}dzdz'\Big)^{1-\frac\beta2},
$$
and therefore, using Lemma~\ref{lem:Minko} again,
\begin{equation} 
\expec{\3 X_2 \3_\beta^4}^\frac14 \,\lesssim \, \e^\beta  \|f\|_{L^4(0,1)} \|g\|_{L^4(0,1)} .
\end{equation}
The combination of above estimates allows to turn \eqref{e.decompX1X2-non-int} into the claim~\eqref{e:NA-non-int} (by noticing that every integral in $z$ or $z'$ is taken over $[0,\frac1\e]$, we obtain the factor  $|\log\e|$ while $\beta=1$ by bounding $\omega_c$ with its maximum).

\subsubsection{Quantitative CLT: proof of Theorem~\ref{0:1D-thm-quant3-non-int}}
As in the proof of Theorem \ref{0:1D-thm-quant3} we have (using the same notation)
$$
|\tilde \nu_\e(A)-\mathcal N(A)| \le d_{TV}(\nu_\e,\mathcal N)+d_{TV}(\mathcal N_{\tau_\e},\mathcal N).
$$
So by combining the two propositions above,
$$
d_{TV}\Big(\frac{I_\e(f,g)-\expec{I_\e(f,g)}}{\sqrt\e \sigma(f,g)},\mathcal N \Big) \lesssim \pi_\beta(\e)+\Pi_\beta(\e)
$$
and the theorem follows.

\section{One-dimensional results revisited} \label{0.quant-revisited}

The results established above are not representative of homogenization in general 
in the way they are written. Indeed, not only do their proofs  rely on explicit formulas, but also
their formulations.
The aim of this short section is precisely to reformulate these results in a form that
can be generalized to higher dimensions (although their proofs cannot). 
For higher dimensions, we refer the reader to \cite{GNO-reg,GNO-quant,DGO1,DGO2,duerinckx2019scaling}, which address the case of Gaussian coefficient fields (however uniformly elliptic and bounded). 
Similar results for log-normal coefficients as considered here (however with integrable covariance) are the object of \cite{CGQ}.

\medskip

We start with a reformulation of \eqref{0:1D-6} in Theorem~\ref{0:1D-thm-qual2}
\begin{equation}\label{0:1D-6revisited}
\lim_{\e \downarrow 0} u_\e'(x)-\bar u'(x)(1+\phi'(\tfrac x\e))\,=\,0, \quad \phi'(y):=\bar a (\frac1{a(y)}-\frac1{\bar a}).
\end{equation}
The random field $\phi$ (uniquely defined by choosing $\phi(0)=0$ almost surely)
is what is called the \emph{corrector} (it corrects the values of $\bar u'(x)$ to reconstruct $u'_\e(x)$).
The object $x\mapsto \bar u^{2s}_\e(x):=\bar u(x)+\e \bar u'(x) \phi(\frac x\e)$ is called the two-scale expansion of $u_\e$, and it satisfies $\lim_{\e \downarrow 0} \|\bar u_\e^{2s}-u_\e\|_{H^1(0,1)}=0$ almost surely. Two-scale expansions can be formally extended to arbitrary order in form of $\sum_{j=0}^n \e^j \frac{d^j}{dx^j} \bar u(x) \phi_j(\frac x\e)$, where $\phi_j$ would denote the corrector of order $j$ -- whereas this makes perfect sense for a periodic coefficient $a$, the validity of higher-order two-scale expansions in the random setting strongly depends on dimension (the higher the dimension, the higher-order the two-scale expansion) and on the correlations of $a$ -- see \cite{BLP-78,MR4000840}.

A direct computation shows that $\phi$ satisfies the equation 
$(a(y)(1+\phi'(y))'\,=\,0$ on $\R$, that is, the version for $d=1$ of the set of $d$ equations
\begin{equation}\label{e.corr-eq-gen}
-\nabla \cdot a(y)(\Id+\nabla \phi(y))\,=\,0.
\end{equation}
This \emph{corrector equation} is central to the theory of homogenization in higher dimension.
Whereas in dimension $d=1$, the solution vanishing at zero is explicitly given by $x \mapsto \int_0^x \bar a (\frac1{a(y)}-\frac1{\bar a})dy$, the solvability of \eqref{e.corr-eq-gen} is subtle in higher dimension. The two properties of $\phi$ that one can retain from dimension 1 are: stationarity of the gradient ($\phi'(y)=\bar a (\frac1{a(y)}-\frac1{\bar a})$ and $\frac1a$ is stationary) with finite second moment, and sublinearity at infinity in the sense that $\lim_{\e \downarrow 0} \|\e \phi(\frac x\e)\|_{L^2(0,1)}=0$ almost surely. These two conditions ensure the almost sure well-posedness of \eqref{e.corr-eq-gen} (cf.~\cite{PapaVara,Kozlov-79,JKO94}).

In dimension $d=1$, $\phi'$ is bounded if $a$ is bounded and isolated from zero -- this is not the case for log-normal coefficients, for which boundedness fails and is replaced by the strong stochastic integrability given by Lemma~\ref{lem:Minko}. 
In higher dimension, $\nabla \phi$ is never a bounded field (even for bounded and uniformly elliptic $a$), and the only control we have for ergodic coefficients is $\expec{|\nabla \phi|^2} <\infty$. For  Gaussian coefficients in higher dimensions, corrector gradients also have strong stochastic integrability -- see \cite[Theorems~2 and 3]{GNO-reg}.

In our one-dimensional Gaussian setting, we also have (sharp) estimates for $\phi$ itself: For all $p\ge 1$ and $x\in \R$,
\[
\expec{|\phi(x)|^p}^\frac1p \lesssim \begin{cases}
|x|^{1-\frac{\beta}{2}},&\text{if $0<\beta<1$},\\
\sqrt{|x|} |\log(2+|x|)|^\frac12,&\text{if $\beta=1$},\\
\sqrt{|x|} ,&\text{if $\beta>1$},
\end{cases}
\]
as a consequence of the proof of Theorem~\ref{0:1D-thm-quant1}. These quantitative estimates of the growth of the corrector in the Gaussian setting extend to higher dimensions (the growth then depends both on $\beta$ and on $d$, and for $\beta>2$ and $d>2$ moments of the correctors even remain bounded wrt $x\in \R^d$ -- cf \cite[Theorem~2]{GNO-quant}).
It is interesting to understand why these estimates depend on dimension.
In dimension $d=1$, $\phi(x)=|x| \fint_0^x \bar a (\frac1{a(y)}-\frac1{\bar a})dy$, that is, $\phi(x)$ is $|x|$ times the average between $0$ and $x$ of the random field $y\mapsto \bar a (\frac1{a(y)}-\frac1{\bar a})$ which has vanishing expectation. Since this average vanishes by the ergodic theorem when $|x|\uparrow +\infty$, the bound on the growth of the corrector is due to stochastic cancellations, and can be quantified by how fast averages of the corrector gradient go to zero.
Although the fondamental theorem of calculus is used here to pass from the corrector gradient to the corrector itself, one can replace that step by some use of potential theory in higher dimensions, and the growth of correctors can also be quantified by how fast averages of their gradients go to zero -- which depends on dimension (think of the central limit theorem scaling).

The \emph{homogenized coefficients} $\bar a$ (a $d\times d$ matrix) can be written in terms of the corrector in form of 
\begin{eqnarray}
\bar a &=& \expec{a(0)(\Id+\nabla\phi(0))}\label{e.homfo-corr}
\\
&\stackrel{d=1}{=}&\expec{\tfrac1a}^{-1}.\label{e.homfo-1d}
\end{eqnarray}
In higher dimensions, formula \eqref{e.homfo-corr} remains valid, but \eqref{e.homfo-1d} is only true in dimension 1 (there is no explicit formula for $\bar a$ in general for $d>1$). 
Let $D$ be some bounded domain in $\R^d$ and  $f\in L^2(D)$. For all $\e>0$, consider the unique weak solution in $u_\e \in H^1_0(D)$ of  $-\nabla \cdot a(\frac \cdot \e) \nabla u_\e=f$ and the unique weak solution in $\bar u \in H^1_0(D)$ of  the homogenized problem $-\nabla \cdot \bar a \nabla \bar u=f$.
Qualitative stochastic homogenization, in form of the almost sure weak convergence of $u_\e$ to $\bar u$ in $H^1_0(D)$, was first established in \cite{PapaVara,Kozlov-79,JKO94} under the assumption that $a$ is uniformly bounded and elliptic, stationary and ergodic.
The qualitative two-scale expansion result \eqref{0:1D-6revisited} then holds in form of the almost sure convergence
$\lim_{\e \downarrow 0} \|(\Id+\nabla \phi(\cdot/\e) )\nabla \bar u-\nabla u_\e\|_{L^2(D)}=0$ -- which, in the stationary ergodic setting, follows from a slight modification of the proof of \cite[Proof of Theorem 3 and Proposition 1]{GNO-quant}.
 
In our one-dimensional Gaussian setting, the qualitative two-scale expansion result \eqref{0:1D-6revisited} was made quantitative in Theorem~\ref{0:1D-thm-quant1}.
Such results (which are based on the growth of the corrector) can be extended to higher dimensions for Gaussian coefficients -- see \cite[Theorem~2]{GNO-quant}. 

\medskip

We conclude by revisiting the results on fluctuations, and define the \emph{homogenizaton commutator} (a $d\times d$-tensor constructed
starting from $a$, $\bar a$, and $\nabla \phi$, and introduced in \cite{DGO1} in this context)
\begin{eqnarray}
\Xi(y)&:=&(a(y)-\bar a)(\Id+\nabla \phi(y)) \\
&\stackrel{d=1}{=}& \bar a-\frac{\bar a^2}{a(y)}.\nonumber
\end{eqnarray}
Let us further study this quantity in dimension $d=1$ for $\beta>1$.
For all $\psi\in L^2(0,1)$, define $J_\e(\psi):=\int_0^1 \Xi(\frac x\e) \psi(x)dx$.
Proceeding as in the previous section, we obtain the same scaling as in Theorem~\ref{0:1D-thm-quant2}:
$$
\var{J_\e(\psi)} \,\lesssim\, \e \|\psi\|^2_{L^2(0,1)}.
$$
In higher dimension, we would assume $\beta>d$, and $\e$ would be replaced by $\e^d$ (the square of the CLT scaling -- see \cite[Proposition~1]{DGO2}).
Even more directly than in the proof of Theorem~\ref{0:1D-thm-quant3}, we also get
\begin{equation}\label{0:eq-pathwise4}
d_{TV}\Big(\frac{J_\e(\psi)}{\sqrt\e \tilde \sigma(\psi)},\mathcal N \Big) \,\lesssim\, \sqrt \e \exp(C|\log \e|^\frac12)+ \e^{\beta-1},
\end{equation}
(note that $\expec{J_\e(\psi)}=0$) where the limiting variance is given this time by
\begin{equation} 
\tilde \sigma^2(\psi):= \bar a^4 \calQ  \int_0^1 \psi^2.
\end{equation}
In higher dimension, one has a result of a similar form (see \cite[Theorem~1]{duerinckx2019scaling}).
The last observation is the following identity
\begin{multline}\label{0:eq-pathwise1}
(a(\tfrac x\e)-\bar a) u'_\e(x)- (a(\tfrac x\e)-\bar a)(1+\phi'(\tfrac x\e)) \bar u'(x)
\\
=\,(1-\frac{\bar a}{a(\tfrac x\e)})\Big( \big(\int_0^1 \frac1{a(\tfrac y\e)}dy\big)^{-1}\int_0^1 \frac1{a(\tfrac y\e)}f(y)dy-\int_0^1 f(y)dy \Big).
\end{multline}
Notice that the left-hand side can be interpreted as the homogenization commutator of the solution minus its two-scale expansion using the standard homogenization commutator -- in the spirit of \eqref{0:1D-6revisited}. 
We shall prove that, when integrated against a test-function, the left-hand side is smaller than the fluctuation scaling.   
Set $\bar g:=\int_0^1 g$, and define $\bar v$ as the unique solution of 
$$
\bar a \bar v''=g' \text{ on }(0,1), \ \bar v(0)=\bar v(1)=0.
$$
On the one hand, by definition of $J_\e$ and since $\bar v'=\frac1{\bar a}(g-\bar g)$,
\begin{equation}
\frac1{\bar a}\int_0^1 (a(\tfrac x\e)-\bar a)(1+\phi'(\tfrac x\e)) \bar u'(x)(g(x)-\bar g)dx = J_\e(\bar u' \bar v'),
\end{equation}
and by definition of $I_\e$, integrating equation \eqref{0:1D-1}, setting $f(x)=\int_0^x f'(y)dy$, $\bar f= \int_0^1 f$, and recalling
the value \eqref{0:1defC1} of the integration constant $C^1_\e$ 
\begin{eqnarray}
\frac1{\bar a}\int_0^1 (a(\tfrac x\e)-\bar a) u'_\e(x)(g(x)-\bar g)&=& \frac1{\bar a}\int_0^1 (f(x)+C^1_\e)(g(x)-\bar g)dx-I_\e(f,g-\bar g) \nonumber
\\
&=&\frac1{\bar a}\int_0^1 (f(x)-\bar f)(g(x)-\bar g)dx- I_\e(f,g), \label{0:eq-pathwise2}
\end{eqnarray}
since $I_\e(f,g-\bar g)=\int_0^1 u_\e'(g-\bar g)=I_\e(f,g)+\bar g \int_0^1 u_\e'=I_\e(f,g)$ using that $u_\e(0)=u_\e(1)=0$.
On the other hand, the random variable obtained by integrating the RHS of \eqref{0:eq-pathwise1} against $\frac1{\bar a}(g-\bar g)=\bar v'$,
that is,
\begin{equation*}
K_\e(f,g):=\Big( \big(\int_0^1 \frac1{a(\tfrac y\e)}dy\big)^{-1}\int_0^1 \frac1{a(\tfrac y\e)}f(y)dy-\int_0^1 f(y)dy \Big) \int_0^1 (\frac1{\bar a}-\frac{1}{a(\tfrac x\e)})(g(x)-\bar g)dx
\end{equation*}
satisfies (using a similar argument as in the proof of Theorem~\ref{0:1D-thm-quant2})
\begin{equation}\label{0:eq-pathwise3}
\var{K_\e(f,g)} \,\lesssim \,\e^2 \|f\|_{L^4(0,1)}\|g\|_{L^4(0,1)}.
\end{equation}
Note that the scaling in $\e$ of \eqref{0:eq-pathwise3} is one order higher than the scaling of  Theorem~\ref{0:1D-thm-quant2}.
In particular, \eqref{0:eq-pathwise1}--\eqref{0:eq-pathwise3} imply that 
\begin{equation}\label{0:eq-pathwise5}
\expec{|\e^{-\frac12}I_\e(f,g)-\e^{-\frac12}J_\e(\bar u' \bar v')|^2}^\frac12 \,\lesssim\, \sqrt{\e},
\end{equation}
whereas \eqref{0:eq-pathwise4} yields
\begin{equation}\label{0:eq-pathwise6}
d_{TV}\Big(\frac{J_\e(\bar u' \bar v')}{\sqrt\e \tilde \sigma(\bar u' \bar v')},\mathcal N \Big) \,\lesssim\, \sqrt \e \exp(C|\log \e|^\frac12)+ \e^{\beta-1},
\quad \tilde \sigma^2(\bar u' \bar v')= \bar a^4 \calQ \int_0^1 \bar u'^2 \bar v'^2.
\end{equation}
Let us reformulate the variance using the explicit form \eqref{0:1D-3} of $\bar u'=\frac1{\bar a}(f- \bar f)$ and $\bar v'=\frac1{\bar a}(g-\bar g)$:
\begin{equation*}
\tilde \sigma^2(\bar u' \bar v')\,=\, \bar a^4 \calQ \int_0^1 \bar u'^2 \bar v'^2
\,=\,  \calQ \int_0^1 (f-\bar f)^2 (g-\bar g)^2\,=\, \sigma^2(f,g),
\end{equation*}
as defined in Theorem~\ref{0:1D-thm-quant3}.
The combination \eqref{0:eq-pathwise5} and \eqref{0:eq-pathwise6} allows to obtain a (slightly weaker) variant of Theorem~\ref{0:1D-thm-quant3},
where the total variation distance is replaced by a distance on probability measures that is controlled by the $L^1$-norm in probability,
e.g.~the 1-Wasserstein distance:
\begin{eqnarray*}
d_{W}\Big(\frac{I_\e(f,g)}{\sqrt\e \sigma(f,g)},\mathcal N \Big) & \le& d_{W}\Big(\frac{J_\e(\bar u'\bar v')}{\sqrt\e \sigma(f,g)},\mathcal N \Big)+\frac1{\sqrt{\e}\sigma(f,g)} \expec{|I_\e(f,g)-J_\e(\bar u'\bar v')|}
\\
&\lesssim& \sqrt{\e},
\end{eqnarray*}
provided $\calQ>0$.

\medskip

Estimate~\eqref{0:eq-pathwise5} illustrates the \emph{pathwise structure of fluctuations}: The quantity $I_\e(f,g)-J_\e(\bar u'\bar v')$ is small in a strong norm in probability (the second moment, as opposed to
weak convergence of measures) in the scaling of the fluctuations (that is, after multiplying by $\sqrt{\e}^{-1}$, this difference is still small).
This implies that the fluctuations of $\e^{-\frac12} I_\e(f,g)$ are described at leading order by the fluctuations of the commutator $\e^{-\frac12}J_\e(\bar u' \bar v')$, a simpler quantity (it only depends on $a$, the corrector, 
and the solution of the homogenized equation). These results extend to higher dimensions, see \cite[Theorem~1]{DGO2}.
Estimate~\eqref{0:eq-pathwise6} constitutes a quantitative CLT for the homogenization commutator, which is actually easier to establish than the corresponding result for $I_\e(f,g)$: The commutator is a more local function of $a$ (in this one-dimensional setting, it is even exactly local: $\Xi(x)=\bar a- \frac{\bar a^2}{a(x)}$ --- compare to \eqref{0:1D-16}). In higher dimension, this result corresponds to \cite[Theorem~1]{duerinckx2019scaling}.

\section*{Acknowledgements}
The authors acknowledge financial support from the  European Research Council (ERC) under the European Union's Horizon 2020 research and innovation programme (Grant Agreement n$^\circ$~864066).


\def\cprime{$'$}

\end{document}